\documentclass[11pt,a4paper]{scrartcl}

\usepackage[utf8]{inputenc}
\usepackage[T1]{fontenc}
\usepackage[english]{babel}
\usepackage{verbatim}
\usepackage{abstract}
\usepackage{url}
\usepackage[hidelinks]{hyperref}

\usepackage[pdftex]{graphicx}
\usepackage{latexsym}
\usepackage{amsmath,amssymb,amsthm, mathtools,tikz}
\usetikzlibrary{cd}
\usepackage{pifont}

\usepackage{makecell}

\usepackage{caption}

\usepackage[nottoc]{tocbibind}
\usepackage{aliascnt}
\usepackage{slashed}

\newcommand{\res}{\textup{res}}
\newcommand{\ita}[1]{\textit{#1}}

\newcommand*\diff{\mathop{}\!\textup{d}}
\newcommand{\col}{\colon \thinspace}

\newcommand{\Vol}{\textup{Vol}}
\newcommand{\Hom}{\textup{Hom}}
\newcommand{\End}{\textup{End}}
\newcommand{\Ad}{\textup{Ad}}

\newcommand{\hol}[1]{{C^{{#1},\alpha}}}
\newcommand{\VertC}[1]{\Vert_{C^{{#1}}}}
\newcommand{\VertH}[1]{\Vert_{C^{{#1},\alpha}}}
\newcommand{\im}{\textup{Im}\thinspace }

\newcommand{\cmark}{\ding{51}}
\newcommand{\xmark}{\ding{55}}

\newcommand{\G}{\textup{G}}
\newcommand{\SO}{\textup{SO}}
\newcommand{\GL}{\textup{GL}}

\newcommand{\Sp}{\textup{Sp}}
\newcommand{\U}{\textup{U}}
\renewcommand{\Im}{\textup{Im}\thinspace}
\renewcommand{\Re}{\textup{Re}\thinspace}
\renewcommand{\H}{\textup{H}}

\DeclarePairedDelimiter{\Abs}{\|}{\|}

\title{Coassociative submanifolds in Joyce's generalised Kummer constructions}
\author{Dominik Gutwein}

\numberwithin{equation}{section}

\newcommand{\mynewtheorem}[2]{
  \newaliascnt{#1}{dummycounter}
  \newtheorem{#1}[#1]{#2}
  \aliascntresetthe{#1}
  \expandafter\def\csname #1autorefname\endcsname{#2}
}

\theoremstyle{plain}
\mynewtheorem{theorem}{Theorem}
\mynewtheorem{proposition}{Proposition}
\mynewtheorem{lemma}{Lemma}
\mynewtheorem{corollary}{Corollary}

\theoremstyle{definition}
\mynewtheorem{definition}{Definition}
\mynewtheorem{assumption}{Assumption}
\mynewtheorem{notation}{Notation}
\mynewtheorem{example}{Example}

\theoremstyle{remark}
\mynewtheorem{remark}{Remark}

\addto\extrasenglish{

}

\begin{document}
\maketitle

\begin{abstract}
This article constructs coassociative submanifolds in $\G_2$-manifolds arising from Joyce's generalised Kummer construction. The novelty compared to previous constructions is that these submanifolds all lie within the critical region of the $\G_2$-manifold in which the metric degenerates. This forces the volume of the coassociatives to shrink to zero when the orbifold-limit is approached.
\end{abstract}

\noindent \textbf{Changes to the published version:} This article is a revision of \cite{Gutwein-coassociatives_in_Kummer} published in Pure and Applied Mathematics Quarterly Volume 20 (2024) Number 2 . The current version corrects some typographical errors and mathematical inaccuracies and provides additional explanations and clarifications. Moreover, it rectifies two minor mistakes in \cite[Example~3.2 and Example~4.9]{Gutwein-coassociatives_in_Kummer} (see \autoref{ex:coassociative_model} and \autoref{Ex:Reidegeld_Dic3} below).

\section{Introduction}\label{sec:intro}

Associative and coassociative submanifolds are the natural subobjects in $7$-dimensional $\G_2$-manifolds. Besides having minimal volume among all submanifolds realising a fixed homology class (and therefore being minimal, cf.~\cite[Sections~2.4 and~4.1.A-B]{HarveyLawson}), they play a prominent role in the extensively studied gauge theory on $\G_2$-manifolds (see for 
example~\cite{Tian-Gauge+Calibrations} and~\cite{DonaldsonSegal-Gauge}). Moreover, Halverson and Morrison proposed that associative and coassociative submanifolds might play a role in characterising the period domain of a $\G_2$-manifold~\cite[Section~3]{HalMor-GaugeEnhancement} (see also the formulation in~\cite[Introduction]{DPSDTW-Associatives}). More precisely, assume that $Y$ is a simply-connected and  compact $7$-manifold that admits torsion-free $\G_2$-structures. In analogy to the Kähler cone of a Calabi--Yau 3-fold, Halverson and Morrison~\cite[Section~3]{HalMor-GaugeEnhancement} proposed that (the $\G_2$-period domain)
\begin{align*}
\mathcal{Q}(Y) \coloneqq \big\{ ([\phi],[*_\phi \phi]) \in \H^3(Y) \oplus \H^4(Y) \mid &\textup{ $\phi \in \Omega^3(Y)$ is a torsion-free $\G_2$-structure}\big\} 
\end{align*}
might be fully characterised by the following inequalities:\footnote{Where we ignore for the moment the issue that the notions of $\G_2$-instantons, associative-, and coassociative submanifolds themselves depend on $\phi$. Furthermore, note that the integration is carried out with respect to the orientation determined by $\frac{1}{7}[\phi] \cup [*_{\phi}\phi]\in \H^7(Y)$}
\begin{enumerate}
\item A topological condition: $\int_Y \alpha\wedge\alpha \wedge \phi < 0$ for every nonzero $[\alpha] \in \H^2(Y)$. 
\item A characteristic class condition: $\int_Y p_1(E) \wedge \phi < 0$ for any vector bundle $E$ over $Y$ admitting a non-flat $\G_2$-instanton.
\item~\label{bul:CalCon} Two calibrated submanifold conditions: 
\begin{itemize}
\item $\int_P \phi > 0$ for any associative submanifold $P$. 
\item $\int_M *\phi>0$ for any coassociative submanifold $M$.
\end{itemize}
\end{enumerate}

If  Halverson and Morrison's proposal is indeed true, then certain degenerations of $\G_2$-structures would be detectable by the vanishing of one of the above integrals. As a step towards this proposal, Dwivedi, Platt, and Walpuski constructed therefore 
in~\cite{DPSDTW-Associatives} associative submanifolds in families of $\G_2$-manifolds arising from Joyce's generalised Kummer 
construction~\cite{Joyce-GeneralisedKummer1,Joyce-GeneralisedKummer2}. These associative submanifolds have the property that their volume shrinks to zero as the $\G_2$-manifold approaches its (singular) orbifold-limit. (This is equivalent to $\int_P \phi_t \to 0$ where $P$ denotes the mentioned associative and $\phi_t$ corresponds to the degenerating path of $\G_2$-structures.) The purpose of the article at hand is to augment their work by the analogous construction of coassociative submanifolds. We hereby proceed as follows:

In \autoref{Sec:background} we review the necessary background on the generalised Kummer construction and asymptotically locally Euclidean (ALE) hyperkähler $4$-manifolds. \autoref{sec:Perturbing_coassociatives} is devoted to the analysis of our construction. In \autoref{theo:perturbing_Coassociatives} we prove a perturbation result for coassociative submanifolds whose spirit is well-known from gluing constructions in gauge theory. It roughly states that whenever two closed $\G_2$-structures $\phi$ and $\phi_0$ on a $7$-manifold $Y$ are 'close' (in a quantified sense) to one another, then a $\phi_0$-coassociative submanifold $M$ can be perturbed to a $\phi$-coassociative, provided that $[\phi_{\vert M}] = 0 \in H^3(M)$. In \autoref{prop:coassociatives_examples} we prove that this theorem is applicable to a certain class of submanifolds that occur very frequently in generalised Kummer constructions. These submanifolds are modelled on (or covered by) the product of a $2$-torus and a holomorphic sphere where the latter lies in the exceptional divisor of the glued-in ALE hyperkähler $4$-manifold appearing in the Kummer construction (cf. \autoref{ex:coassociative_model}). Subsequently, we find in \autoref{Sec:Examples} numerous examples of coassociative submanifolds in various resolutions of $\G_2$-orbifolds constructed 
in~\cite{Joyce-GeneralisedKummer2} and~\cite{Reidegeld}.  Our construction leaves the freedom of choosing a 'basepoint' of the 2-torus and we mention in \autoref{rem:coassfam} that moving this basepoint produces the full deformation family of the coassociatives constructed in \autoref{prop:coassociatives_examples}. In our examples, this family is either homeomorphic to $S^1$ or a closed interval. The coassocative submanifolds in the latter case are embedded for the inner values of the interval and factor at the endpoints through a double-cover over an (embedded) rigid coassociative submanifold. Ultimately, we give in \autoref{sec:ReidegeldResolution} all choices of ALE hyperkähler $4$-manifold that can be used to resolve the $\G_2$-orbifolds 
of~\cite{Reidegeld} that were treated in \autoref{Sec:Examples}.

There already exists a vast literature on the construction of coassociative submanifolds (see~\cite[Chapter~12]{Joyce-Blue} and~\cite[Section~6]{Lotay-Calibrated} for an overview). Here we only mention that Joyce~\cite[Section~4.2]{Joyce-GeneralisedKummer2} constructed coassociative submanifolds inside his generalised Kummer constructions as fixed-point sets of anti $\G_2$-involutions. At least one part of their support lies outside the critical region of the ambient manifold in which the orbifold singularities develop. In contrast, the coassociatives in the article at hand are all constructed to lie completely within this region. This is ultimately the reason why their volume shrinks to zero.

\subsection*{Acknowledgements}

I am grateful to my PhD-supervisor Thomas Walpuski for suggesting this problem to me. Many ideas in this article arose from countless meetings we had which are invaluable to me. Furthermore, I am indebted to Gorapada Bera and Viktor Majewski for their constructive tips, discussions, and proofreading. Ultimately, I would like to thank Dominic Joyce and Daniel Platt for answering my questions 
on~\cite{Joyce-CalabiYau} and~\cite{Platt-Estimates}, respectively, and an anonymous referee for helpful suggestions. This material is based upon work supported by the Simons Collaboration “Special Holonomy in Geometry, Analysis, and Physics”.

\section{Background}\label{Sec:background}

\subsection{Joyce's generalised Kummer construction}\label{Sec:generalisedKummer}

The generalised Kummer construction, as developed (and extended) by Joyce 
in~\cite{Joyce-GeneralisedKummer1,Joyce-GeneralisedKummer2,Joyce-Black}, produces compact manifolds with holonomy contained in $\G_2$ as desingularisations of certain $\G_2$-orbifolds. This section follows the presentation 
in~\cite{DPSDTW-Associatives} very closely. The following class of examples serve as models for the singularities considered in this article:
\begin{example}\label{ex_G2Model}
Let $(X,\underline{\omega})$ be a hyperkähler 4-orbifold with hyperkähler structure $\underline{\omega} \in \Omega^2(X,\im \mathbb{H}^*)$. Denote by $\Vol \in \Omega^3(\im \mathbb{H})$ and $\underline{\sigma} \in \Omega^1(\im \mathbb{H},\im \mathbb{H})$ the volume form and the canonical isomorphism $T\im \mathbb{H} \to \im \mathbb{H} \times \im \mathbb{H}$, respectively. In the following we denote by $\langle \underline{\sigma} \wedge \underline{\omega} \rangle$ the $3$-form on $\im \mathbb{H} \times X$ obtained by wedging and pairing $\im \mathbb{H} \otimes \im \mathbb{H}^* \to \mathbb{R}$.
\begin{enumerate}
\item The product $\im \mathbb{H} \times X$ carries a torsion-free $\G_2$-structure defined by
\begin{equation}\label{eq:modelG2Structure}
\phi \coloneqq \Vol - \langle \underline{\sigma} \wedge \underline{\omega} \rangle \in \Omega^3(\im \mathbb{H} \times X).
\end{equation}
\item Assume there is a group action $\rho \col G \to \textup{Isom}(X)$ by $G< \SO(\im\mathbb{H}) \ltimes  \im \mathbb{H}$ that preserves the hyperkähler structure in the sense that for any $(R,v)\in G$ 
\begin{equation}\label{eq:HyperkählerEquivariance}
(\rho(R,v)^*\otimes R^*) \underline{\omega} = \underline{\omega}. 
\end{equation}
The $3$-form $\phi$ is invariant under the product action on $\im \mathbb{H} \times X$ and descends to a torsion-free $\G_2$-structure on the quotient $Y\coloneqq (\im \mathbb{H} \times X)/G$. We denote the corresponding $3$-form on $Y$ by $\phi$ as well. Note that whenever $G$ is Bieberbach (i.e. discrete, cocompact, and torsion-free) then the action is free and taking the quotient does not introduce additional singularities in $Y$.
\end{enumerate}  
\end{example}

Let now $(Y_0,\phi_0)$ be a compact and flat $\G_2$-orbifold such that its  singularities are locally modelled on $\mathbb{R}^3 \times \mathbb{H}/\Gamma$ for a finite group $\Gamma < \Sp(1)$. More precisely, we demand:
\begin{assumption}\label{Ass:Codim4Singularities}
Denote by $\mathcal{S}$ the set of connected components of the singular set of $Y_0$. We assume that for every $S \in \mathcal{S}$ there exist
\begin{enumerate}
\item A finite subgroup $\Gamma_S < \Sp(1)$, a Bieberbach group $G_S< \SO(\im \mathbb{H}) \ltimes \im \mathbb{H}$, and a group action $\rho \col G_S \to N_{\SO(\mathbb{H})}(\Gamma_S)/\Gamma_S\subset \textup{Isom}(\mathbb{H}/\Gamma_S)$. Denote by \[(Y_S \coloneqq (\im \mathbb{H} \times \mathbb{H}/\Gamma_S)/G_S,\phi_S)\] the corresponding $\G_2$-orbifold from \autoref{ex_G2Model}. 
\item An open set \[ U_S \coloneqq (\im \mathbb{H} \times B_{2R_S}(0)/\Gamma_S)/G_S \subset Y_S \] for $R_S>0$ and an open embedding $\mathtt{J}_S \col U_S \to Y_0$ with $S \subset \mathtt{J}_S(U_S)$ and $\mathtt{J}_S^*\phi_0 = \phi_S$. The $R_S$ are chosen such that $\mathtt{J}_{S_1}(U_{S_1})\cap \mathtt{J}_{S_2}(U_{S_2}) = \emptyset $ for any two $S_1 \neq S_2 \in \mathcal{S}$.
\end{enumerate}
\end{assumption}

\begin{remark}\label{rem:ADE}
All (non-trivial) finite subgroups $\Gamma < \Sp(1)$ were classified by 
Klein~\cite{Klein-ADE}. These are isomorphic to:
\begin{itemize}
\item[$(A_k)$] The cyclic group $C_{k+1}$ for $k\geq1$
\item[$ (D_k)$] The dicyclic group $\textup{Dic}_{k-2}$ for $k\geq 3$
\item[$(E_6)$] The binary tetrahedral group $2T$
\item[$(E_7)$] The binary octahedral group $2O$
\item[$(E_8)$] The binary icosahedral group $2I$
\end{itemize}
(See also~\cite[Section~2]{Reidegeld} for a description on how these groups lie inside $\Sp(1)$.)
\end{remark}

\begin{definition}\label{def:RData}
Let $(Y_0,\phi_0)$ be a flat $\G_2$-orbifold satisfying \autoref{Ass:Codim4Singularities}. A set of resolution data consists for every $S \in \mathcal{S}$ of the following:
\begin{enumerate}
\item \label{def:ALE-space}  An asymptotically locally Euclidean (ALE) hyperkähler manifold asymptotic to $\mathbb{H}/\Gamma_S$. That is, a hyperkähler $4$-manifold $(\hat{X}_S,\underline{\hat{\omega}}_S)$ together with a diffeomorphism $\tau_S \col \hat{X}_S\setminus \hat{K}_S \to (\mathbb{H}\setminus B_{R_S}(0))/\Gamma_S$ outside a compact set $\hat{K}_S \subset \hat{X}_S$ that satisfies \[ \vert \nabla^k \big({\tau_S}_*\underline{\hat{\omega}}_S - \underline{\omega}\big) \vert = \mathcal{O}(r^{-4-k}).\] The norm and covariant derivatives are hereby taken with respect to the flat metric on $(\mathbb{H}\setminus\{0\})/\Gamma_S$.
\item A group action $\rho_S \col G_S \to \textup{Isom}(\hat{X}_S)$ which leaves $\hat{K}_S$ and $\underline{\hat{\omega}}_S$ invariant (in the sense of \eqref{eq:HyperkählerEquivariance}) and makes $\tau_S$ equivariant.
\end{enumerate}
\end{definition}

For a given orbifold $Y_0$, a set of resolution data, and a positive parameter $t>0$ we define the following sets: 
\begin{align*}
V\coloneqq \bigsqcup_{S\in \mathcal{S}} V_S \quad &\textup{for} \quad V_S\coloneqq (\im \mathbb{H} \times B_{R_S}(0)/\Gamma_S)/G_S \subset (\im \mathbb{H} \times \mathbb{H}/\Gamma_S)/G_S \\
U\coloneqq \bigsqcup_{S\in \mathcal{S}} U_S \quad &\textup{for} \quad U_S \coloneqq (\im \mathbb{H} \times B_{2R_S}(0)/\Gamma_S)/G_S \subset (\im \mathbb{H} \times \mathbb{H}/\Gamma_S)/G_S \\
\hat{V} \coloneqq \bigsqcup_{S\in \mathcal{S}} \hat{V}_S \quad &\textup{for} \quad \hat{V}_S \coloneqq (\im \mathbb{H} \times \hat{K}_S)/G_S \subset (\im \mathbb{H} \times \hat{X}_S)/G_S \\
\hat{U}_t \coloneqq \bigsqcup_{S\in \mathcal{S}} \hat{U}^t_S \quad &\textup{for} \quad \hat{U}_S^t \coloneqq( \im \mathbb{H} \times (t \tau_S)^{-1} (B_{2R_S}(0)/\Gamma_S))/G_S \subset (\im \mathbb{H} \times \hat{X}_S)/G_S
\end{align*}

Denote by $\mathtt{J} \col U \to Y_0$ and $t \tau \col \hat{U}_t \to U$ the maps induced by all $\{\mathtt{J}_S\}_{S\in \mathcal{S}}$ and $\{t\tau_S\}_{S \in \mathcal{S}}$, respectively. 

\begin{definition}[{\cite[proof of Theorem~2.2.1]{Joyce-GeneralisedKummer2}}]\label{def:OrbifoldResolution}
Given a flat $\G_2$-orbifold $(Y_0,\phi_0)$ and a set of resolution data, Joyce defines a 1-parameter family of smooth manifolds by \[\hat{Y}_t \coloneqq \big(Y_0 \setminus \mathtt{J}(V)\big) \cup \big(\hat{U}_t \cup \hat{V}\big)/\sim \] where $\hat{U}_t \ni x \sim \mathtt{J}(t\tau(x)) \in \mathtt{J}(U\setminus V)$.

Furthermore, Joyce equips each $\hat{Y}_t$ with a closed $\G_2$-structure $\tilde{\phi}_t$ that has the following property: On any $\hat{V}_S \subset \hat{Y}_t$, $\tilde{\phi}_t$ agrees with the Model Structure~\eqref{eq:modelG2Structure} associated to the (rescaled) hyperkähler structure $t^{2} \underline{\hat{\omega}}_S$ on $\hat{K}_S$.
\end{definition}

\begin{remark}
Instead of working with $\tilde{\phi}_t$ we 
follow~\cite{DPSDTW-Associatives} and work with the rescaled $\G_2$-structure $t^{-3} \tilde{\phi}_t$.
\end{remark}

The following existence theorem was first proven by Joyce 
in~\cite{Joyce-GeneralisedKummer1} and later reproven with improved estimates by Platt 
in~\cite{Platt-Estimates}. The following formulation is taken from~\cite[Theorem~2.19]{DPSDTW-Associatives}.
\begin{theorem}[{\cite[Theorem~4.58]{Platt-Estimates}}]\label{theo:Platt}
Let $(Y_0,\phi_0)$ be a compact and flat $\G_2$-orbifold satisfying \autoref{Ass:Codim4Singularities} and let $\mathcal{R}$ be a set of resolution data. Furthermore, let $\alpha \in (0,1/32)$ be a chosen Hölder exponent. Then there are $T_0 = T_0(\mathcal{R})$ and $c=c(\mathcal{R},\alpha)>0$ such that for any $t\in (0,T_0)$ there exists a torsion-free $\G_2$-structure $\phi_t$ on $\hat{Y}_t$ with $[\phi_t]=[\tilde{\phi}_t] \in \H^3(\hat{Y}_t)$ and \[ \big\Vert t^{-3} (\phi_t - \tilde{\phi}_t) \big\Vert_\hol{1} < c t^{5/2}. \] The $C^{1,\alpha}$-norm above is taken with respect to the metric $t^{-2}\tilde{g}_t$ (induced by $t^{-3} \tilde{\phi}_t$).
\end{theorem}
\begin{remark}
Note that the formulation of \autoref{theo:Platt} in~\cite[Theorem~4.58]{Platt-Estimates} bounds the $C^{1,\alpha}$-norm of $\phi_t-\tilde{\phi}_t$ only by $t^{3/2-\alpha}$. However, if one uses the $C^{1,\alpha}$-norm with respect to $t^{-2}\tilde{g}_t$  (instead of $\tilde{g}_t$) one obtains the estimate in \autoref{theo:Platt} as a direct consequence of the estimate with respect to the weighted norm given in~\cite[Theorem~4.58]{Platt-Estimates}.
\end{remark}
\begin{remark}
The formulation of \autoref{theo:Platt} 
in~\cite{Platt-Estimates} only considers $\G_2$-orbifolds whose singularities are resolved via Eguchi--Hanson spaces. Its proof relies on the property that the set
\[ \big\{ \omega\in \Omega^2(X_{\textup{EH}}) \mid \Delta \omega = 0 \textup{ and } \vert \nabla^\ell \omega \vert = \mathcal{O}(r^{\beta-\ell}) \textup{ for each } \ell \in \mathbb{N}_0 \big\} \]
is independent of $\beta \in [-4,0)$. It was explained to us by Thomas Walpuski that this property holds for every ALE 4-manifold. One way to prove this is as~\cite[Proposition~5.10]{Walpuski-InstantonsKummer} using the improved Kato inequality for harmonic 2-forms in~\cite[Theorem~1]{Seaman-Harmonic_two_forms_4D}. The proof of \autoref{theo:Platt} given 
in~\cite{Platt-Estimates} adapts therefore to resolutions of $\G_2$-orbifolds by arbitrary ALE hyperkähler $4$-manifolds.
\end{remark}

\subsection{Asymptotically locally Euclidean hyperkähler 4-manifolds}\label{sec:ALE}

Recall from \autoref{def:RData} that a resolution of a flat $\G_2$-orbifold requires the choice of an ALE hyperkähler 4-manifold together with a lift of the action of the Bieberbach group. In the following section we review how these can be constructed. All these spaces contain holomorphic spheres which are the main ingredient in our construction of coassociative submanifolds later in this article. 

Note that \autoref{Sec:KronheimerALE} is only relevant for \autoref{Ex:Reidegeld_Dic3} and may be skipped at the reader's preference.
\subsubsection{The Gibbons--Hawking Ansatz}\label{sub:GibbonsHawking}

For $N\in \mathbb{N}$ let $C_N<\Sp(1)$ be the cyclic subgroup generated by right-multiplication with $e^{2\pi i /N}$. Concrete models of ALE spaces asymptotic to $\mathbb{H}/C_N$ were first constructed for $N=2$ by Eguchi and 
Hanson~\cite{EguchiHanson} and then by Gibbons and 
Hawking~\cite{GibbonsHawkings} for general $N$. A detailed treatment of the following material can be found in~\cite[Section~59]{TW-RiemGeo} (see also~\cite[Remark~2.12]{DPSDTW-Associatives} and~\cite[Section~3.5]{Gibbons-Hyperkaehler}): 

\begin{enumerate}
\item For any \[ \zeta \in \Delta \coloneqq \{ [\zeta_1,\dots,\zeta_N] \in (\im \mathbb{H})^N/S_N \mid \textup{$\zeta_1 + \dots + \zeta_N = 0$} \} \] define $Z_\zeta \coloneqq \{\zeta_1,\dots,\zeta_N\} \subset \im \mathbb{H}$, $B_\zeta \coloneqq \im \mathbb{H} \setminus Z_\zeta$, and $f_\zeta \in C^\infty(B_\zeta)$ by \[f_\zeta (q) \coloneqq \sum_{a=1}^N \frac{1}{2 \vert q-\zeta_a\vert }.\] 

The function $f_\zeta$ is a sum of harmonic functions and one can furthermore check that the cohomology class $ [*_3\diff f_\zeta] $ lies inside the image of the canonical inclusion $ H^2(B_\zeta,2 \pi \mathbb{Z}) \hookrightarrow H^2(B_\zeta,\mathbb{R})$. This implies that there exists a (up to isomorphism) unique principal $\U(1)$-bundle $\pi_\zeta \col X^\circ_\zeta \to B_\zeta$ together with a connection $1$-form $i \theta \in \Omega^1(X_\zeta^\circ,i \mathbb{R})$ that satisfies $\diff \theta = \pi_\zeta^*(*_3 \diff f_\zeta)$. 

For $\zeta_i \in Z_\zeta$ denote by $N_{\zeta_i}$ the number of entries of $\zeta$ equal to $\zeta_i$. Around any sphere $S^2 \subset B_\zeta$ whose inner ball only contains $\zeta_i \in Z_\zeta$, the restriction $(X_\zeta^\circ)_{\vert S^2}$ is isomorphic to the quotient of the Hopf-fibration by $C_{N_{\zeta_i}}$.

\item The Gibbons--Hawking Ansatz defines a hyperkähler structure on the total space $X^\circ_\zeta$ as follows: The connection induces a horizontal distribution $X^\circ_\zeta \times \im \mathbb{H} \subset TX_\zeta^\circ$. Furthermore, we identify the vertical tangent bundle $X^\circ_\zeta \times i \mathbb{R}$ with $X^\circ_\zeta \times \Re \mathbb{H}$ via $(x,it) \mapsto (x,t/f_\zeta(x))$. This induces a canonical hypercomplex structure $\underline{I}_\zeta$ on $TX_\zeta^\circ \cong X_\zeta^\circ \times \mathbb{H}$ which is compatible with the metric defined by \[g_\zeta^\circ \coloneqq f^{-1}_\zeta \cdot \theta \otimes \theta + f_\zeta \cdot \pi_\zeta^* (g_{\im \mathbb{H}}).\] The corresponding hyperhermitian form $\underline{\omega}_\zeta$ is closed and therefore hyperkähler.

\item It turns out that $(X_\zeta^\circ,\underline{\omega}_\zeta)$ can be extended to a complete hyperkähler orbifold $(X_\zeta,\underline{\omega}_\zeta)$ by adding $\# Z_\zeta$ points, one over each element in $Z_\zeta$. In fact, whenever \[ \zeta \in \Delta^\circ \coloneqq \{ [\zeta_1,\dots,\zeta_N] \in \Delta \mid \textup{$\zeta_i\neq \zeta_j$ for $i \neq j$} \}, \] then $X_\zeta$ is a manifold. 

Outside a ball $B_{R^2}(0)$ containing all of $Z_\zeta$, the bundle $(X_\zeta)_{\vert\im \mathbb{H} \setminus B_{R^2}(0)}$ has Chern class $-N \in \mathbb{Z} \cong H^2(\im \mathbb{H}\setminus B_{R^2}(0),\mathbb{Z}).$ It is therefore isomorphic to the principal $\U(1)$-bundle
\begin{align*}
\pi_0 \col (\mathbb{H}\setminus B_R(0))/C_N &\to \im \mathbb{H}\setminus B_{R^2}(0) \\
[q] &\mapsto qi\bar{q}.
\end{align*}
With the right choice of such an isomorphism $\tau_\zeta$ (e.g. using parallel transport in radial direction and 'matching' the connections $\theta_\zeta$ and $\theta_0$ at the sphere at infinity) one can show that $\underline{\omega}_\zeta$ approaches the standard hyperkähler structure on $\mathbb{H}/C_N$ as in \autoref{def:RData}, \autoref{def:ALE-space}. The Gibbons--Hawking spaces are therefore ALE asymptotic to $\mathbb{H}/C_N$.

\item \label{bull:Lambda^2_+} Let $R \in N_{\SO(\mathbb{H})}(C_N)$. Identify\footnote{via $\im \mathbb{H}\ni\xi \mapsto \langle\underline{\omega},\xi \rangle \in \Omega^2(\mathbb{H})$ where $\underline{\omega} = \tfrac{-1}{2} \diff q \wedge \diff \bar{q} \in \Omega^2(\mathbb{H},(\im \mathbb{H})^*) $ is the standard hyperkähler structure on $\mathbb{H}$.} the space of self-dual $2$-vectors $\Lambda^2_+\mathbb{H}$ with $\im \mathbb{H}$ and denote by $\Lambda^2_+ R \in \SO(\im \mathbb{H})$ the induced map. Furthermore, define
\begin{align*}
\alpha_R \coloneqq 
\begin{cases}
1, \text{ if $R \in Z_{\SO(\mathbb{H})}(C_N)$} \\
-1, \text{ else}
\end{cases}
\end{align*}
where $Z_{\SO(\mathbb{H})}(C_N)$ denotes the centralizer of $C_N$ in $\SO(\mathbb{H})$. If $\zeta\in \Delta$ satisfies $ \Lambda^2_+R \zeta = \alpha_R \zeta$, then there exists an $\hat{R}\in \textup{Isom}(X_\zeta)$ satisfying \[(\hat{R}^* \otimes \Lambda^2_+R^*)\underline{\omega}_{\zeta} = \underline{\omega}_{\zeta} \quad \textup{and} \quad \tau_{\zeta} \circ \hat{R} = R \circ \tau_{\zeta} \] where $R$ acts on $\mathbb{H}/C_N$ in the obvious way. This is explained in more detail in \autoref{Sec:KronheimerALE}, \autoref{bul:ALE_group_action}. Note, however, that whenever $R\in N_{\SO(\mathbb{H})}(C_N)$ for $N\geq 3$ (which holds in all examples of \autoref{Sec:Examples}), then $\hat{R}$ can be constructed by lifting $\alpha_R \Lambda^2_+R \in \textup{O}(\im \mathbb{H})$ to a bundle (anti-) isomorphism on $X_\zeta \setminus \pi_\zeta^{-1}(Z_\zeta)$ and requiring that $\tau_\zeta \circ \hat{R} = R \circ \tau_\zeta$ and $\hat{R}^*(i \theta_\zeta) = \pm i \theta_\zeta$.
\item Let $\zeta_0 \neq \zeta_1 \in Z_\zeta$ and assume that the line segment \[\ell \coloneqq \{ t \zeta_{1} + (1-t) \zeta_{0} \mid t \in [0,1] \} \subset \Im \mathbb{H}\] intersect $Z_\zeta$ only in its endpoints. The preimage $\Sigma_\ell \coloneqq \pi_\zeta^{-1}(\ell)\subset X_\zeta$ is a smoothly embedded sphere, which is holomorphic with respect to the complex structure $I_{\zeta,\hat{\xi}}\coloneqq \langle \underline{I}_\zeta,\hat{\xi}\rangle$ for $\hat{\xi} \coloneqq \frac{\zeta_{1}-\zeta_0}{\vert \zeta_{1}-\zeta_0 \vert}\in \im\mathbb{H}$. 

Let now $R\in N_{\SO(\mathbb{H})}(C_N)$ satisfy $\Lambda^2_+R \zeta = \alpha_R \zeta$ and denote by $\hat{R}$ its lift to $X_\zeta$ (as described in \autoref{bull:Lambda^2_+} above). Then $\hat{R}(\Sigma_\ell) = \Sigma_{\alpha_R\Lambda^2_+R (\ell)}$, where $\alpha_R\Lambda^2_+R (\ell)$ denotes the line segment coming from applying $\alpha_R \Lambda^2_+R \in \textup{O}(\im \mathbb{H})$ to $\ell$. 
\end{enumerate}

\subsubsection{Kronheimer's construction of ALE spaces}\label{Sec:KronheimerALE}

All hyperkähler ALE $4$-manifolds asymptotic to $\mathbb{H}/\Gamma$ for any finite subgroup $\Gamma < \Sp(1)$ were constructed and classified by Kronheimer 
in~\cite{Kronheimer-ALE1} and~\cite{Kronheimer-ALE2}. The following summary follows the one given in~\cite[Remark~2.15]{DPSDTW-Associatives}. Note also that for $\Gamma =C_N$ this treatment is equivalent to \autoref{sub:GibbonsHawking}.
\begin{enumerate}
\item Let $\mathbb{C}[\Gamma] \coloneqq \textup{Maps}(\Gamma,\mathbb{C})$ denote the regular representation equipped with its standard Hermitian inner product. Furthermore, define \[ S \coloneqq (\mathbb{H} \otimes_\mathbb{R} \mathfrak{u}(\mathbb{C}[\Gamma]))^\Gamma \quad \text{and} \quad G\coloneqq \mathbb{P}\U(\mathbb{C}[\Gamma])^\Gamma\] and equip $S$ with the canonical flat hyperkähler structure. The adjoint action of $G$ on $S$ has a distinguished hyperkähler moment map \[\mu \col S \to (\im \mathbb{H})^*\otimes \mathfrak{g}^*.\] Let $\mathfrak{z}^* \subset \mathfrak{g}^*$ be the annihilator of $[\mathfrak{g},\mathfrak{g}]$, i.e. all elements in $\mathfrak{g}^*$ fixed by the coadjoint action of $G$. For any value $\zeta \in (\im \mathbb{H})^*\otimes \mathfrak{z}^*$, the hyperkähler quotient $X_\zeta \coloneqq \mu^{-1}(\zeta)/G$ is a hyperkähler orbifold asymptotic to $\mathbb{H}/\Gamma$ (\cite[Lemma~3.3 and Proposition~3.14]{Kronheimer-ALE1}).
\item \autoref{rem:ADE} associates a root system $\Phi$ to $\Gamma$. Kronheimer~\cite[Proposition~4.1]{Kronheimer-ALE1} defines an isomorphism between $\mathfrak{z}^*$ and the associated Cartan algebra $\mathfrak{h} \coloneqq (\mathbb{R}\Phi)^*$. For any root $\theta \in \Phi$ let $D_\theta \coloneqq \ker \theta \subset \mathfrak{h}$ be the associated wall of the Weyl chambers. If \[\zeta \in \tilde{\Delta}^\circ \coloneqq ((\im \mathbb{H})^* \otimes \mathfrak{h}) \setminus \cup_{\theta \in \Phi} ((\im \mathbb{H})^*\otimes D_\theta), \] then $X_\zeta$ is a manifold~(\cite[Proposition~2.8]{Kronheimer-ALE1}).
\item \label{bul:ALE_group_action} Any $R\in N_{\SO(\mathbb{H})}(\Gamma)$ acts on $\Gamma$ by conjugation. We extend this to a complex linear map $C_R \in \U(\mathbb{C}[\Gamma])$. The standard representation of $R$ on $\mathbb{H}$ tensored with the Adjoint action of $C_R$ on $\mathfrak{u}(\mathbb{C}[\Gamma])$ induces an action on $S$. The hyperkähler moment map satisfies \[\mu \circ (R\otimes \Ad_{C_R})= (\Lambda^2_+R^{-1} \otimes \Ad_{C_R^{-1}})^* \circ \mu \] where $\Lambda^2_+ R$ is as in \autoref{sub:GibbonsHawking} \autoref{bull:Lambda^2_+} and $\Ad^*_{C_R^{-1}}$ denotes the coadjoint representation of $C_R$ on $\mathfrak{h} \cong \mathfrak{z}^* \subset \mathfrak{g}^*$. Thus, if \[(\Lambda^2_+R\otimes \Ad_{C_R})^* \zeta = \zeta, \] we obtain an induced isometry $\hat{R}\in \text{Isom}(X_{\zeta})$ satisfying \[(\hat{R}^* \otimes \Lambda^2_+R^*)\underline{\omega}_{\zeta} = \underline{\omega}_{\zeta} \quad \textup{and} \quad \tau_{\zeta} \circ \hat{R} = R \circ \tau_{\zeta}. \]

\item\label{bul:description-Ad_C_R-in-Kronheimer} Here is one way to understand the coadjoint action of $\Ad_{C_R}^*$ on $\mathfrak{h}$ in the previous paragraph (see also~\cite[Section~3]{Joyce-CalabiYau}): First, we identify $\mathfrak{h} \cong \mathbb{R}\Phi$ via the inner product and  let $\{\theta_1,\dots,\theta_n\}\subset \Phi$ be a set of simple roots. Next, let $\{(R_1,\rho_1),\dots,(R_n,\rho_n)\}$ the set consisting of (representatives of) all isomorphism classes of irreducible non-trivial (complex) representations of $\Gamma$. The 
McKay-Correspondence~\cite{McKay-Graphs-Sing-finiteGroups} gives rise to a distinguished bijection between these two sets (cf.~\cite[Section~2]{Kronheimer-ALE1}). 

$\Ad^*_{C_{R}}$ acts on $\{\alpha_1,\dots,\alpha_n\} \cong \{(R_1,\rho_1),\dots,(R_n,\rho_n)\}$ by mapping $(R_i,\rho_i)$ to the irreducible representation $(R_j,\rho_j) \cong (R_i,\rho_i \circ C_{R})$ (where we precompose the representation with conjugation by $R \in N_{\SO(\mathbb{H})}(\Gamma)$). The map $\Ad^*_{C_{R}} \col \mathfrak{h}\to \mathfrak{h}$ is the linear extension of this action.\footnote{This can be seen when using the isomorphism $\tau\col \mathfrak{z}^*\to \mathfrak{h}$ in~\cite[Equation~(2.7)]{Kronheimer-ALE1}. Note further, that~\cite[Proposition~4.1]{Kronheimer-ALE1} implies that $\tau$ and the isomorphism in~\cite[Section~4]{Kronheimer-ALE1} only differ by a conformal transformation.} 

\item \label{bul: holomorphic curves in Kronheimer's ALE} Let $\zeta \in \tilde{\Delta}^\circ$ be fixed and $\theta \in \Phi$ be a root. Define $\xi \in \im \mathbb{H}$ by \[ \langle \xi,\cdot \rangle = \zeta(\theta)\in (\im \mathbb{H})^*\] and let $\hat{\xi} \coloneqq \xi/\vert \xi\vert$. Inside $X_\zeta$ lies a nodal Riemann surface $\Sigma_\theta$ which is holomorphic with respect to the complex structure $I_{\zeta,\hat{\xi}}\coloneqq \langle \underline{I}_\zeta, \hat{\xi} \rangle$.  

If $\theta_1,\theta_2 \in \Phi$ are two roots such that $\theta = \theta_1 + \theta_2$ and $\vert \zeta(\theta) \vert = \vert \zeta(\theta_1)\vert + \vert \zeta(\theta_2)\vert $, then $\Sigma_\theta$ is the union of the (nodal) $I_{\hat{\xi}}$-holomorphic curves $\Sigma_{\theta_1}$ and $\Sigma_{\theta_2}$ attached along one new pair of nodes. If no decomposition with this property exists, then $\Sigma_\theta$ is itself an embedded $2$-sphere.

Let now $R\in N_{\SO(\mathbb{H})}(\Gamma)$ satisfy $(\Lambda^2_+R \otimes \Ad_{C_R})^* \zeta = \zeta$ and denote by $\hat{R}$ its lift as described in \autoref{bul:ALE_group_action}. This isometry maps $\Sigma_\theta$ to the surface $\hat{R}(\Sigma_\theta)=\Sigma_{\Ad_{C_R}(\theta)}$ where $\Ad_{C_R}$ denotes the adjoint action on $(\mathbb{R}\Phi) \cong \mathfrak{z}\subset \mathfrak{g}$ (as described in \autoref{bul:description-Ad_C_R-in-Kronheimer}).

\item \label{bul: Kronheimer Weyl group} Denote by $W$ the Weyl group of $\Phi$. If two elements $\zeta_1,\zeta_2 \in \Delta^{\circ}$ are related by an element in $W$, then $X_{\zeta_1}$ and $X_{\zeta_2}$ are isomorphic as hyperkähler ALE spaces (cf.~\cite[Section~3]{Kronheimer-ALE2} and~\cite[Section~3]{AtiyahBielawski-Nahm}). This isomorphism identifies the holomorphic spheres $\Sigma_\theta \subset X_{\zeta_1}$ and $\Sigma_{w\theta} \subset X_{\zeta_2}$ where $w\in W$ satisfies $\zeta_2 = w \zeta_1$. Furthermore, one can arrange for this isomorphism to intertwine the asymptotic coordinates $\tau_{\zeta_1}$ and $\tau_{\zeta_2}$. We can therefore replace $\zeta \in \tilde{\Delta}^\circ$ in the previous discussion by \[ [\zeta] \in \Delta^\circ \coloneqq \tilde{\Delta}^\circ /W.\]
\end{enumerate}

\section{Perturbing coassociative submanifolds}\label{sec:Perturbing_coassociatives}

Throughout this section, $(Y,\phi)$ denotes a $7$-manifold equipped with a closed $\G_2$-structure.
\begin{definition}[{\cite[Corollary~IV.1.20]{HarveyLawson}}]
A $4$-dimensional immersed submanifold $\iota \col M \to Y$ is called coassociative if $\iota^*\phi = 0$. If we want to emphasize the underlying $\G_2$-structure we will write $\phi$-coassociative.
\end{definition}
\begin{example}\label{ex:coassociative_model}
Let $(X,\underline{\omega})$ be a hyperkähler $4$-manifold together with an action $\rho \col G \to \textup{Isom}(X)$ by a Bieberbach group $G$. We denote the corresponding $\G_2$-manifold from \autoref{ex_G2Model} by $(Y,\phi)$. Furthermore, note that the normal subgroup $\Lambda \coloneqq G \cap \im \mathbb{H} < \im \mathbb{H}$ is a lattice. An immersed coassociative submanifold inside of $Y$ can now be constructed from the following data:
\begin{enumerate}
\item An embedded Riemann surface $\iota_\Sigma \col \Sigma \to X$ which is holomorphic with respect to $I_{\hat{\xi}_1} = \langle \underline{I},\hat{\xi}_1 \rangle$ for $\hat{\xi}_1 \in S^2 \subset \im \mathbb{H}$.
\item Two linearly independent $\xi_2,\xi_3 \in \big\{\hat{\xi}_1\big\}^\perp \cap \Lambda \subset \Im \mathbb{H}$ such that $\rho(\xi_2)(\Sigma)= \Sigma =\rho(\xi_3)(\Sigma)$. 
\end{enumerate}
Furthermore, we require the choice of basepoint $q\in \im \mathbb{H}$. We then define \[ M \coloneqq ((\mathbb{R}\xi_2+\mathbb{R}\xi_3) \times \Sigma) / \langle\xi_2,\xi_3\rangle_\mathbb{Z} \] and $\iota_q\col M \to Y$ as $\iota_q \big( [y,z] \big) \coloneqq [q+y,\iota_\Sigma(z)].$ It immediately follows from~\eqref{eq:modelG2Structure} that $\iota_q$ is a coassociative immersion. Next, we discuss conditions under which $\iota_q$ is an embedding or factors through a covering map over an embedding. For this we assume that 
\begin{enumerate}
\setcounter{enumi}{2}
\item $\xi_2,\xi_3 \in \Lambda$ generate the sublattice $\Lambda \cap (\mathbb{R}\xi_2 + \mathbb{R}\xi_3)$ and $\rho(\Lambda)=\{1\}$. \label{bul:cond_embedding_1}
\end{enumerate}
We can then regard $T^2_q \coloneqq [q] + (\mathbb{R}\xi_2 + \mathbb{R} \xi_3)/\langle \xi_2,\xi_3\rangle_\mathbb{Z} \subset (\im \mathbb{H})/\Lambda$ as an embedded submanifold. Assume further that 
\begin{enumerate}
\setcounter{enumi}{3}
\item \label{bul:cond_embedding_2} $G/\Lambda \cong H_1 \times H_2$ where $H_1,H_2$ are (possibly trivial) groups that satisfy the following:
\begin{enumerate}
\item The only element $h\in H_1$ with $\rho(h)(\Sigma)\cap \Sigma \neq \emptyset$ and $h\cdot T_q \cap T_q \neq \emptyset$ is $h=1$. Note here and below that $G/\Lambda$ canonically acts on $X$ as $\rho(\Lambda)=\{1\}$.
\item Every $h\in H_2$ satisfies $\rho(h)(\Sigma)= \Sigma$ and $h\cdot T_q = T_q$.
\end{enumerate} 
\end{enumerate}
In this case, the (free) $H_2$ action lifts to $M$ and $\iota_q$ descends to an embedding $\overline{\iota_q} \col M/H_2 \to Y$. Note that the conditions in \autoref{bul:cond_embedding_2} (and the groups $H_1,H_2$) depend on the choice of $q$.
\end{example}
\begin{remark}\label{rem:modelcoassoc_families}
By varying the chosen basepoint $q\in \im \mathbb{H}$ in the construction of $\iota_q$ in the previous example as well as the $I_{\hat{\xi}_1}$-holomorphic embedding $\iota_\Sigma \col \Sigma \to X$, one produces a (up to reparametrisation) $(1+b_1(\Sigma))$-dimensional family of coassociative immersions (one dimension comes from varying $q$ and $b_1(\Sigma)$ from varying $\iota_\Sigma$). This is of course in accordance with~\cite[Theorem~4.5]{McLean} which implies that the moduli space of coassociative immersions $\iota \col M \to Y$ is itself an orbifold of dimension $b^2_+(M) = 1 +b_1(\Sigma)$.
\end{remark}

It is well-known (cf.~\cite[Proposition~4.2]{McLean}) that for a coassociative immersion $\iota\col M \to Y$, the mapping \[\iota^*TY \ni v \mapsto \iota^*(i_v\phi) \in \Lambda^2T^*M \] descends to an isomorphism between the normal bundle and $\Lambda^2_+T^*M$ (the bundle of self-dual $2$-forms).

\begin{definition}
A tubular neighbourhood of the coassociative immersion $\iota \col M \to Y$ is a convex open neighbourhood $U\subset \Lambda^2_+T^*M$ of the zero section together with an open immersion $\mathtt{J}\col U \to Y$ which restricts to $\iota$ at the zero section. Additionally, we demand that for any $u \in U$ the image of $\partial_t (\mathtt{J}(t u))_{\vert t=0} \in \iota^*TY$ in $\Lambda^2_+T^*M$ under the isomorphism described above is again $u$. 
\end{definition}
\begin{remark}
Subsequently, we may simply use the tubular neighbourhood induced by the Levi--Civita connection of the ambient manifold $Y$.
\end{remark}
Let now $\mathtt{J}\col U \to Y$ be a tubular neighbourhood. For any $\omega \in \Gamma(U)$ we denote by $\mathtt{J}_\omega \col M \to Y$ the immersion $x \mapsto \mathtt{J}(\omega_x)$. Furthermore, we define $F_\mathtt{J}\col \Gamma(U) \to \Omega^3(M)$ by $F_\mathtt{J}(\omega) = \mathtt{J}_\omega^*\phi$.
By definition, the immersed submanifold $\mathtt{J}_\omega\col M \to Y$ for $\omega \in \Gamma(U)$ is coassociative if and only if $F_\mathtt{J}(\omega) =0$.
\begin{proposition}[{\cite[Theorem~4.5]{McLean}}]\label{prop:linearisation}
Let $\mathtt{J} \col U\subset \Lambda^2_+M \to Y$ be a tubular neighbourhood of a coassociative immersion $\iota \col M \to Y$. Then the map $F_\mathtt{J}$ has image contained in $\diff \Omega^2(M)$. Furthermore, there exists  a smooth map $\mathcal{N}_\mathtt{J} \in C^{\infty}(\Gamma(U),\diff\Omega^2(M))$ that satisfies \[F_\mathtt{J}(\omega) = \diff \omega + \mathcal{N}_\mathtt{J}(\omega) \] and such that for each $k \in \mathbb{N}_0$ and $\alpha\in (0,1)$ there is a constant $c = c(\mathtt{J},k,\alpha)>0$ with \[ \Vert \mathcal{N}_\mathtt{J}(\omega)-\mathcal{N}_\mathtt{J}(\eta) \Vert_{C^{k,\alpha}} \leq c (1+R) (\Vert \omega \Vert_\hol{k+1} + \Vert \eta \Vert_\hol{k+1} ) \Vert \omega - \eta \Vert_\hol{k+1}  \] for any $\omega,\eta \in \Gamma(U)$ with $\Vert \omega\Vert_\hol{k+1},\Vert \eta\Vert_\hol{k+1} \leq R$.
\end{proposition}

The proof of this proposition except the estimate on $\mathcal{N}_\mathtt{J}$ can be found in~\cite[proof of Theorem~4.5]{McLean}. As the arguments are short, we have included them here for the reader's convenience.
\begin{proof}
Since $F_\mathtt{J}(0) =0$ and the cohomology class doesn't change under homotopies, we have that $[F_\mathtt{J}(t\omega)]=0 \in \H^3(M)$ for every $t\in [0,1]$. This proves the first point.

For the second point we observe that the Fundamental Theorem of Calculus and $F_\mathtt{J}(0)=0$ imply \[F_\mathtt{J}(\omega) = D_0 F_\mathtt{J} (\omega) + \int_0^1 \partial_t F_\mathtt{J}(t\omega) - D_0 F_\mathtt{J}(\omega) \diff t\] where $D_0F_\mathtt{J}$ denotes the linearisation of $F_\mathtt{J}$ at the zero section. It therefore remains to check that $D_0F_\mathtt{J}(\omega)$ equals $\diff \omega$ and \[\mathcal{N}_\mathtt{J}(\omega) \coloneqq \int_0^1 \partial_t F_\mathtt{J}(t\omega) - D_0F_\mathtt{J}(\omega)\diff t \] satisfies the quadratic estimate.

For every point $x \in M$ there exists a vector field $v \in  \Gamma(TY)$ such that in an open neighbourhood around $x$ we have $\omega = \iota^*(i_v\phi)$ and $\varphi^v_t \circ \iota=\mathtt{J}_{t\omega}$ for $t\in (-\varepsilon,\varepsilon)$ where $\varphi^v$ denotes the flow of $v$. Since $\phi$ is closed, we obtain around $x$: \[ D_0F_\mathtt{J}(\omega) = \partial_t (\iota^*(\varphi^v_t)^*\phi)_{\vert t=0} = \iota^* \diff (i_v \phi) = \diff \omega. \]

\noindent The estimate for $\mathcal{N}_\mathtt{J}$ is standard but rather lengthy and can be found in \autoref{Sec:quadraticestimate}.
\end{proof}
\begin{theorem}\label{theo:perturbing_Coassociatives}
Let $\alpha \in (0,1)$ be a fixed Hölder-exponent and $\beta,\gamma,c,R>0$ be constants with $\beta > 2 \gamma$. Then there are $T,c_v > 0$ depending only on $\beta,\gamma,c,R$ with the following significance: Let $\phi,\phi_0$ be two closed $G_2$-structures on $Y$ and $\iota \col M \to Y$ be an immersed $\phi_0$-coassociative submanifold with tubular neighbourhood $\mathtt{J} \col U \subset \Lambda^2_+M \to Y$ that satisfy  
\begin{enumerate}
\item $B_R(0)\subset U$
\item $\iota^*[\phi]=0\in \H^3(M)$
\item $\Vert \mathtt{J}^*(\phi-\phi_0) \VertH{1} \leq c t^\beta$
\item \label{item:lineat_estimate} $\diff \col (\mathcal{H}^2_+)^\perp \subset  \Omega^2_+(M) \to \Omega^3(M)$ satisfies $\Vert \omega \VertH{2} \leq c t^{-\gamma} \Vert \diff \omega \VertH{1}$, where $(\mathcal{H}^2_+)^\perp$ denotes the $L^2$-orthogonal complement of the space of harmonic self-dual $2$-forms
\item $\Vert \mathcal{N}_\mathtt{J}(\omega)-\mathcal{N}_\mathtt{J}(\eta) \Vert_{C^{1,\alpha}} \leq c (\Vert \omega \Vert_\hol{2} + \Vert \eta \Vert_\hol{2} ) \Vert \omega - \eta \Vert_\hol{2} $
\end{enumerate}
for some $t \in (0,T)$. Then there is a unique section $\omega \in \Gamma(U)\cap (\mathcal{H}^2_+)^\perp$ with $\Vert \omega \VertH{2} \leq c_v t^{\beta-\gamma}$ (where $c_v>0$ is determined in the proof) such that $\mathtt{J}_\omega$ is $\phi$-coassociative.
\end{theorem}
The analogue statement for associative submanifolds can be found in~\cite[Proposition~3.19]{DPSDTW-Associatives} and its proof carries over with only minor adaptations. We have included it here for the convenience of the reader.
\begin{proof}
To ease notation we drop the subscript $\mathtt{J}$ and instead denote by $F_{(0)}(\omega) \coloneqq \mathtt{J}_\omega^*\phi_{(0)}$. Since $\diff_{\vert \Omega^2_+(M)} \col \Omega^2_+(M) \to \diff \Omega^2(M)$ is surjective and $\textup{image}(F_{(0)})\subset \diff \Omega^2(M)$ by \autoref{prop:linearisation} and the second assumption, we can define \[ E(\omega) \coloneqq \diff_{\vert (\mathcal{H}^2_+)^\perp\cap \Omega^2_+(M)}^{-1} (F_0(\omega) - F(\omega) - \mathcal{N}_0(\omega)). \] By our assumptions there is a positive constant $c_E=c_E(c,R)$ such that for every $r \in (0,R)$ and $\omega,\eta \in \overline{B_r(0)} \subset C^{2,\alpha}\Gamma(U)$ the following two inequalities hold:
\begin{align*}
\Vert E(0) \VertH{2} &\leq c_E t^{\beta-\gamma} \\
\Vert E(\omega) - E(\eta) \VertH{2} &\leq  c_E (r+t^\beta) t^{-\gamma} \Vert \omega - \eta \VertH{2}.
\end{align*}
Therefore, $E$ restricts to a contraction on $\overline{B_r(0)}$ provided that \[ c_E (r+t^\beta)t^{-\gamma} < 1 \qquad \text{and} \qquad c_E t^{\beta-\gamma} + c_E (r+t^\beta) t^{-\gamma}r \leq r.\] This holds if we choose $T = T(\beta,\gamma,c,R)$ sufficiently small and for $t\in (0,T)$ the radius $r \coloneqq 2 c_Et^{\beta-\gamma}$. 

\noindent Let now $\omega\in \overline{B_r(0)}$ be the unique fixpoint of $E$. By definition, this satisfies \[0 = \diff \omega + \mathcal{N}_0 (\omega) -F_0(\omega) + F(\omega) = F(\omega) \] and gives therefore rise to a $\phi$-coassociative submanifold (of regularity $C^{2,\alpha}$). For sufficiently small $T$ this section and the corresponding submanifold are smooth by elliptic regularity (cf.~\cite[Proposition~7.16]{Lotay-CSCoassociatives_desing})
\end{proof}

\begin{remark}\label{rem:EmbCoa}
If $M$ in the previous theorem is compact and $\iota \col M \to Y$ an embedding, then $\mathtt{J}_\omega$ will also be an embedding once $t$ is sufficiently small.
\end{remark}

\subsection{The linear estimate for surface fibrations over tori}

The following subsection establishes \hyperref[item:lineat_estimate]{Assumption~\ref*{item:lineat_estimate}} of \autoref{theo:perturbing_Coassociatives} in the case of \autoref{ex:coassociative_model}. We quickly review the relevant set-up: Let $\Sigma$ be a closed Riemann surface equipped with a Riemannian metric $g_\Sigma$ and $\xi_2,\xi_3 \in \mathbb{R}^2$ be linearly independent. Furthermore, assume that there is a group action $\rho \col \langle \xi_2,\xi_3 \rangle_\mathbb{Z} \to \textup{Isom}(\Sigma)$. Our coassociative submanifold in \autoref{ex:coassociative_model} was then defined as \[M = (\mathbb{R}^2 \times \Sigma )/\langle\xi_2,\xi_3\rangle_\mathbb{Z} \] equipped with the induced metric coming from $g_\Sigma$ and $g_{\mathbb{R}^2}$. We will also need the following rescaled version: \[M_t \coloneqq (\mathbb{R}^2 \times \Sigma )/\langle t^{-1}\xi_2,t^{-1}\xi_3\rangle_\mathbb{Z} \] for $t>0$ where $\rho_t \col \langle t^{-1}\xi_2,t^{-1}\xi_3 \rangle_\mathbb{Z} \to \textup{Isom}(\Sigma)$ is given by $\rho(t\cdot)$. The induced metric on $M_t$ is denoted by $g_t$.

Observe that the natural projection $p_t \col M_t \to T^2_t \coloneqq \mathbb{R}^2/\langle t^{-1}\xi_2,t^{-1}\xi_3\rangle_\mathbb{Z}$ gives rise to a fiber bundle. The orthogonal complement $H_t\coloneqq V_t^\perp$ of its vertical tangent bundle $V_t = \ker(Dp_t)$ defines a flat Ehresmann connection. This induces a decomposition $\Omega^\ell(M_t) = \oplus_{p+q=\ell} \Omega^{p,q}$ with $\Omega^{p,q}\coloneqq \Gamma(\Lambda^p H^* \otimes \Lambda^q V^*)$. Furthermore, the operator $\diff + \diff^* \col \Omega^k(M_t) \to \Omega^{k+1}(M_t) \oplus \Omega^{k-1}(M_t)$ splits into 
\[ \diff_{H} + \diff^*_{H} \col \Omega^{p,q} \to \Omega^{p+1,q} \oplus \Omega^{p-1,q} \quad \textup{and} \quad \diff_V + \diff_V^* \col \Omega^{p,q} \to \Omega^{p,q+1} \oplus \Omega^{p,q-1}. \]
\begin{definition} We define the following operators acting on $\Omega^\ell(M_t)$:
\begin{enumerate}
\item Denote by $\Pi_t\in \End( \Omega^{\ell}(M_t))$ the $L^2$-projection onto $\ker(\diff + \diff^*)$.
\item For any $y \in T^2_t$, let $\res_y \col \Omega^{p,q}(M_t) \to \Lambda^pT^*_yT^2_t\otimes \Omega^q(p_t^{-1}(y))$ be the composition 
\begin{align*}
\Omega^{p,q}(M_t) \to \Gamma(p^{-1}_t(y),\Lambda^p H^* \otimes \Lambda^q V^*) \cong \Lambda^p T^*_yT^2_t \otimes \Omega^q(p_t^{-1}(y)).
\end{align*} 
\item The operator $\diff_V + \diff_V^*$ restricts for every $y \in T^2_t$ to an elliptic operator on $\Lambda^pT^*_yT^2_t \otimes \Omega^q(p_t^{-1}(y))$. Denote by $\pi_y$ the $L^2$-orthogonal projection onto its kernel.
\item Finally, denote by $\hat{\pi}\in \End(\Omega^\ell(M_t))$ the operator which maps $\omega \in \Omega^\ell(M_t)$ to the unique $\hat{\pi}(\omega)\in \Omega^\ell(M_t)$ with $\res_y\hat{\pi}(\omega) \coloneqq \pi_y(\res_y\omega)$ for every $y \in T^2_t$.
\end{enumerate}
\end{definition}

\begin{remark}
In all examples of \autoref{Sec:Examples} the fiber bundle $M_t = T^2_t \times \Sigma$ is trivial. In this case $\Omega^{p,q}\cong \Omega^p(T^2_t,\Omega^q(\Sigma))$ and $\diff_V + \diff_V^*$ becomes $\diff_\Sigma + \diff_\Sigma^*$ acting upon $\Omega^q(\Sigma)$. The operator $\hat{\pi} \col \Omega^{p,q} \to \Omega^{p,q}$ is then simply the $L^2$-projection onto $\ker(\diff_\Sigma+\diff_\Sigma^*)$. Furthermore, $\diff_H + \diff_H^* = \diff_{T_t^2} + \diff^*_{T_t^2}$.
\end{remark}

\noindent The main result of this section is the following Fredholm estimate:
\begin{proposition}\label{prop:linear_estimate}
For every $\alpha \in (0,1), k\geq 1$ there is a constant $c=c(k,\alpha,M_1,g_1)$ such that for every $t\in \mathbb{R}^+$ and $\omega \in \Omega^\ell(M_t)$, \[ \Vert \omega \VertH{k} \leq c  ((1+t^{-1})\Vert \diff \omega + \diff^* \omega \VertH{k-1} + \Vert \Pi_t \omega \VertH{k}). \]
\end{proposition}

For this we use the following results on harmonic forms on $\mathbb{R}^2\times \Sigma$ and $M_t$ which are an immediate consequence of~\cite[Lemma~A.1]{Walpuski-InstantonsKummer}.
\begin{lemma}[{\cite[Corollary~4.13]{Platt-Estimates}}]\label{lem:harmonic_forms}
Every harmonic $\omega \in \Omega^\ell(\mathbb{R}^2\times \Sigma)$ with $\Vert \omega \VertC{0} < \infty$ is a sum of terms of the form $\eta_1\otimes \eta_2$, where $\eta_1 \in \Omega^p(\mathbb{R}^2)$ is constant and $\eta_2 \in \Omega^q(\Sigma)$ is harmonic. Identifying the space of constant forms on $\mathbb{R}^2$ with $\Lambda^*\mathbb{R}^2$ we therefore have \[ \mathcal{H}^{\ell}(\mathbb{R}^2\times \Sigma) \cap C^{0}\Omega^\ell(\mathbb{R}^2\times \Sigma) = \bigoplus_{p+q} \Lambda^p\mathbb{R}^2 \otimes \mathcal{H}^q(\Sigma).\]
\end{lemma}

\begin{corollary}\label{cor:harmonic_forms_M_t}
The pull-back of the quotient map $q_t \col \mathbb{R}^2 \times \Sigma \to M_t$ induces an isomorphism \[ \mathcal{H}^\ell (M_t) \cong \bigoplus_{p+q=\ell} \Lambda^p \mathbb{R}^2 \otimes \mathcal{H}^q(\Sigma)^{\langle \xi_2,\xi_3 \rangle_\mathbb{Z}}, \] where $\mathcal{H}^*(\Sigma)^{\langle \xi_2,\xi_3\rangle_\mathbb{Z}}$ denotes the space of harmonic forms on $\Sigma$ invariant under the action of $\langle \xi_2,\xi_3 \rangle_\mathbb{Z}$ by the pull-back of $\rho$.
\end{corollary}

The next two lemmas prove \autoref{prop:linear_estimate} for elements which respectively lie inside and orthogonal to the kernel of $\diff_V+\diff_V^*$.

\begin{lemma}\label{lem:estimate_not_kernel}
For every $\omega\in \Omega^\ell(M_t)$
\[ \Vert (1-\hat{\pi}) \omega \VertH{k} \leq c_2 \Vert (\diff + \diff^*)(1-\hat{\pi})\omega \VertH{k-1}\] holds independently of $t$. It even holds on $\mathbb{R}^2 \times \Sigma$.
\end{lemma}
\begin{proof}
We prove the estimate on $\mathbb{R}^2 \times \Sigma$. Since the quotient maps are isometries, this implies the lemma. 

Suppose the estimate does not hold on $\mathbb{R}^2\times\Sigma$ to produce a contradiction. Then we find a sequence $(\omega_n)_{n\in \mathbb{N}}\subset \Omega^\ell(\mathbb{R}^2\times \Sigma)$ with \[\Vert (1-\hat{\pi}) \omega_n \VertH{k} = 1\quad \textup{and} \quad\Vert (\diff + \diff^*) (1-\hat{\pi})\omega_n \VertH{k-1} \to 0.\] Since both expressions are invariant under translations, we can assume that 
\begin{align}\label{eq:lbound_estimate_not_kernel}
\Vert (1-\hat{\pi}) \omega_n \Vert_{C^{k,\alpha}(B_1(0)\times \Sigma)} \geq \frac{1}{4(k+1)} \qquad \textup{for every $n \in \mathbb{N}$.}
\end{align}

By the Arzelà-Ascoli Theorem we find a subsequence (again denoted by $(\omega_n)_{n \in \mathbb{N}}$) such that $((1-\hat{\pi})\omega_n)_{n \in \mathbb{N}}$ converges in $C^{k-1}_{\textup{loc}}$ to $\omega_\infty\in C^{k-1}\Omega^\ell (\mathbb{R}^2 \times \Sigma)$. This limit satisfies $\diff \omega_\infty + \diff^* \omega_\infty = 0$ (for $k=1$ in the distributional sense) and is therefore smooth by elliptic regularity. As $\Vert \omega_\infty \VertC{k-1} \leq1$, \autoref{lem:harmonic_forms} implies that $\omega_\infty$ is a sum of terms of the form $\eta_1 \otimes \eta_2$ where $\eta_1 \in \Omega^p(\mathbb{R}^2)$ is parallel and $\eta_2\in \Omega^q(\Sigma)$ is harmonic. Therefore, $\omega_\infty = \hat{\pi} \omega_\infty$ and since $\hat{\pi}(1-\hat{\pi})\omega_n = 0$ for every $n\in \mathbb{N}$, we obtain further $\omega_\infty=\hat{\pi}\omega_\infty = 0$. However, bootstrapping improves the convergence inside $B_1(0)\times \Sigma$ to $C^{k,\alpha}$ and therefore $\Vert \omega_\infty \Vert_{C^{k,\alpha}(B_1(0)\times \Sigma)} \geq 1/(4(k+1))$ holds by~\eqref{eq:lbound_estimate_not_kernel}. This gives the sought contradiction.
\end{proof}

\begin{lemma}\label{lem:estimate_kernel}
For every $t \in \mathbb{R}^+$ and $\omega \in \Omega^\ell(M_t)$ we have
\[ \Vert \hat{\pi} \omega \VertC{0} \leq c_3 t^{-1} \Vert (\diff + \diff^*) \hat{\pi} \omega \VertC{0} + \Vert \Pi \omega \VertC{0} \]
where $c_3$ is independent of $t$.
\end{lemma}
\begin{proof}
We first prove the estimate for $t=1$ and then for arbitrary $t$ by scaling.

\noindent \textbf{The estimate for $t=1$} follows from Morrey's inequality and Fredholm theory.

\noindent \textbf{The estimate for general $t \in \mathbb{R}^+$:} Denote by $\Phi_t \col M_1 \to M_t$ the map $[(y,z)] \mapsto [(t^{-1}y,z)]$. One can check that for any $\omega \in \Omega^\ell(M_t)$ we have 
\begin{align*}
\vert\Phi_t^*\omega \vert_{g_1} &= \sum_{p+q=\ell} t^{-p} \big(\vert \omega^{p,q} \vert_{g_t} \circ \Phi_t\big),\\
 (\diff + \diff^*_1) \Phi_t^* \omega  &=  \Phi^*_t \big(\diff \omega + t^{-2} \diff^*_{H_t} \omega + \diff^*_{V_t} \omega \big),  \\
\Pi_1 \Phi_t^* &= \Phi_t^* \Pi_t,
\end{align*}
where $\omega^{p,q}$ denotes the projection onto $\Lambda^p H^* \otimes \Lambda^q V^*$ and where the last equality uses \autoref{cor:harmonic_forms_M_t}. The following estimate uses $\Abs{\cdot}_{C^0_t}$ and $\Abs{\cdot}_{C^0_1}$ to denote the $C^0$-norms with respect to the metrics $g_t$ and $g_1$.
Since the decomposition $\Lambda^\ell T^*M_t = \oplus_{p+q=\ell} \Lambda^p H^* \otimes \Lambda^q V^*$ is orthogonal, the previous step implies that for any $\omega \in \textup{im} \hat{\pi}$:
\begin{align*}
\Vert \omega \Vert_{C^{0}_t} &=  \Abs[\Big]{\sum_{p+q=\ell} \omega^{p,q} }_{C^{0}_t} =  \Abs[\Big]{ \sum_{p+q=\ell} t^p\Phi_t^* \omega^{p,q} }_{C^{0}_1} \\
& \leq c_3  \big( \Abs[\Big]{ (\diff + \diff^*_1) \sum_{p+q=\ell} t^p\Phi^*_t \omega^{p,q} }_{C^{0}_1} + \Abs[\Big]{ \sum_{p+q=\ell} t^p \Pi_1 \Phi_t^*\omega^{p,q} }_{C^{0}_1} \big) \\
& = c_3 \big(  \Abs[\Big]{  \sum_{p+q=\ell} t^{p} \Phi^*_t \diff_{H_t} \omega^{p,q} }_{C^{0}_1} + \Abs[\Big]{  \sum_{p+q=\ell} t^{p-2} \Phi^*_t \diff^*_{H_t} \omega^{p,q} }_{C^{0}_1}  + \Abs[\Big]{ \Pi_t \omega }_{C^{0}_t} \big)\\
& = c_3 \big(  \Abs[\Big]{  \sum_{p+q=\ell} t^{-1} \diff_{H_t} \omega^{p,q} }_{C^{0}_t} + \Abs[\Big]{  \sum_{p+q=\ell} t^{-1}  \diff^*_{H_t} \omega^{p,q} }_{C^{0}_t}  + \Abs[\Big]{ \Pi_t \omega }_{C^{0}_t} \big) \\
&= c_3 \big( t^{-1} \Vert (\diff +\diff^*)\omega \Vert_{C^0_t} + \Vert \Pi_t \omega \Vert_{C^{0}_t} \big). \qedhere
\end{align*}
\end{proof}

\begin{proof}[Proof of \autoref{prop:linear_estimate}]
The Schauder estimate \[\Vert \omega \VertH{k} \leq c_4  \big( \Vert (\diff +\diff^*) \omega \VertH{k-1} + \Vert \omega \VertC{0} \big) \] and Lemmas~\ref{lem:estimate_not_kernel} and~\ref{lem:estimate_kernel} imply 
\begin{align*}
\Vert \omega \VertH{k} \leq c_5 (1+t^{-1}) \big( \Vert (\diff + \diff^*) \omega \VertH{k-1} + \Vert (\diff + \diff^*) \hat{\pi} \omega \VertH{k-1} + \Vert \Pi \omega \VertC{0} \big).
\end{align*}
The observation $(\diff + \diff^*)\hat{\pi} = \hat{\pi} (\diff + \diff^*)$ finishes the proof.
\end{proof}

\section{Examples}\label{Sec:Examples}
Let $(Y_0,\phi_0)$ be a flat $\G_2$-orbifold together with a chosen set of resolution data. Denote by $\tilde{\phi}_t$ the closed $\G_2$-structure from \autoref{def:OrbifoldResolution} on the resolution $\hat{Y}_t$ and by $\phi_t$ the torsion-free $\G_2$-structure of \autoref{theo:Platt}. 
\begin{assumption}\label{ass:coassociative_data}
Assume that for some element $\big((\hat{X}_S,\hat{\omega}_S,\tau_S),(\rho_S \col G_S \to \textup{Isom}(\hat{X}_S))\big)$ of the resolution data we have 
\begin{enumerate}
\item An embedded closed surface $\iota_\Sigma \col \Sigma \to \hat{X}_S$ which is holomorphic with respect to $I_{\hat{\xi}_1} = \langle \underline{I},\hat{\xi}_1 \rangle$ for $\hat{\xi}_1 \in S^2 \subset \im \mathbb{H}$.
\item Two linearly independent $\xi_2,\xi_3 \in \big\{\hat{\xi}_1\big\}^\perp \cap \Lambda_S \subset \Im \mathbb{H}$ such that $\rho(\xi_2)(\Sigma)= \Sigma =\rho(\xi_3)(\Sigma)$. Here $\Lambda_S$ is the lattice $G_S \cap \im \mathbb{H}< \im \mathbb{H}$.
\end{enumerate}
\end{assumption}
\begin{proposition}~\label{prop:coassociatives_examples}
For every triple $(\Sigma,\xi_2,\xi_3)$ as in \autoref{ass:coassociative_data} there exists a $T>0$ such that for every choice of basepoint $q\in \im \mathbb{H}$ there is an immersed $\phi_t$-coassociative submanifold $\iota_t \col M \to \hat{Y}_t$ for $t \in(0,T)$. As $t$ approaches $0$, the induced volume on $M$ shrinks to $0$. Furthermore, if $(G_S,\rho_S)$, $(\Sigma,\xi_2,\xi_3)$, and $q$ satisfy \autoref{bul:cond_embedding_1} and~\ref{bul:cond_embedding_2} listed at the end of \autoref{ex:coassociative_model}, then there exists a free group action of $H_2 < G_S/\Lambda_S$ (as specified in \autoref{ex:coassociative_model}) such that $\iota_t$ descends to an embedding of $M/H_2$ once $t$ is sufficiently small.
\end{proposition}
\begin{proof}
Throughout the proof we work with the rescaled $\G_2$-structures $t^{-3} \tilde{\phi}_t$ and $t^{-3} \phi_t$. \autoref{ex:coassociative_model} gives rise to an immersed $(t^{-3}\tilde{\phi}_t)$-coassociative submanifold $\widetilde{\iota}_t \col M \to \hat{Y}_t$. Let $\mathtt{J} \col U \to \hat{Y}_t$ be its tubular neighbourhood induced by the Levi--Civita connection associated to $t^{-2}\tilde{g}_t$. Without loss of generality we may assume that $\mathtt{J}(U) \subset \hat{V}$ where $\hat{V}$ is as in \autoref{def:OrbifoldResolution}.

The compactness of $M$, \autoref{theo:Platt}, and \autoref{prop:linear_estimate} imply that there exist $t$- and $q$-independent constants $c>0$, $R>0$, $\beta\coloneqq 5/2$, and $\gamma \coloneqq 1$ such that the first three assumptions in \autoref{theo:perturbing_Coassociatives} are satisfied. (All $C^{k,\alpha}$-norms are hereby taken with respect to $t^{-2} \tilde{g}_t$ and we tacitly assume $\alpha<1/32$.) Furthermore, one can check that in our set-up the estimates in \autoref{lem:secDerEstimate} are independent of $t$. Thus, by enlarging $c$ if necessary we may assume that the fourth assumption is also satisfied. We therefore, obtain a $(t^{-3}\phi_t)$-coassociative submanifold (or analogously, a $\phi_t$-coassociative submanifold) $\iota_t \col M \to \hat{Y}_t$ contained in $\hat{V}$ that satisfies \[ \Vert \iota_t - \widetilde{\iota}_t \Vert_{C^{2,\alpha}_{t^{-2}\tilde{g}_t}} \leq c_v t^{3/2} \qquad \text{and} \qquad \Vert \iota_t - \widetilde{\iota}_t \Vert_{C^{2,\alpha}_{\tilde{g}_t}} \leq c_v t^{1/2-\alpha}.  \] Direct inspection reveals that with respect to the family of metrics $\tilde{g}_t$ (and therefore also with respect to $g_t$) the fibers of $M$ collapse to points as $t$ tends to $0$. 

If $(G_S,\rho_S)$, $(\Sigma,\xi_2,\xi_3)$, and $q$ satisfy the conditions given in \autoref{bul:cond_embedding_1} and~\ref{bul:cond_embedding_2}, then there exists a free $H_2$-action on $M$ under which $\widetilde{\iota_t}$ and the tubular neighbourhood chosen above are invariant (cf. \autoref{ex:coassociative_model}). By the uniqueness of the section in \autoref{theo:perturbing_Coassociatives}, $\iota_t$ is also invariant under $H_2$ and descends therefore to $\overline{\iota_t} \col M/H_2 \to \hat{Y}_t$. This is just a perturbation of the embedding $\overline{\widetilde{\iota_t}} \col M/H_2 \to \hat{Y}_t$ (cf. \autoref{ex:coassociative_model}) and therefore also an embedding once $t$ is sufficiently small.
\end{proof}
\begin{remark}\label{rem:coassfam}
The previous proposition produces coassociative submanifolds by perturbing the model-immersion from \autoref{ex:coassociative_model}. This requires the choice of a basepoint $q\in \im \mathbb{H}$. By varying this basepoint \autoref{prop:coassociatives_examples} produces a (up to reparametrization) $1$-dimensional family of coassociative immersions.\footnote{To make this statement rigorous one could use a family version of \autoref{theo:perturbing_Coassociatives} as in~\cite[Proposition~3.19]{DPSDTW-Associatives}.} Since $b^2_+(M)=1$ (as $b_1(\Sigma)=0$ for all immersed (holomorphic) Riemann surfaces in ALE hyperkähler 4-manifolds), all coassociative deformations of $\iota_t \col M \to \hat{Y}_t$ are obtained this way (cf. \autoref{rem:modelcoassoc_families}).
\end{remark}

\begin{example}
Joyce~\cite[Examples~7-14]{Joyce-GeneralisedKummer2} constructs seven examples of flat $\G_2$-orbifolds whose respective singular strata are all given by tori: $S=T^3$. More precisely, neighbourhoods of the singularities in all these orbifolds are described by \autoref{ex_G2Model} with $X = \mathbb{H}/\Gamma_S$ for $\Gamma_S = C_2$. The corresponding Bieberbach groups are given by lattices $\G_S = \Lambda_S \cong \mathbb{Z}^3$ (whose exact form are irrelevant for our purpose) and the group actions $\rho_S \col G_S \to \textup{Isom}(\mathbb{H}/\Gamma_S)$ are trivial.

All these singularities can be resolved via Gibbons--Hawking spaces. This requires a choice of $\zeta \in \Delta^\circ$ (cf. \autoref{sub:GibbonsHawking}). To simply find \ita{some} resolution data, any such choice suffices.

In order to construct coassociative submanifolds, we pick for every singular stratum $S$ two elements $\xi_2,\xi_3 \in \Lambda_S$ which generate $\Lambda_S \cap (\mathbb{R}\xi_2 + \mathbb{R}\xi_3)$. Furthermore, we choose $\zeta \coloneqq [-\zeta_1,\zeta_1]\in \Delta^\circ $ with $0\neq\zeta_1 \in \{\xi_2,\xi_3\}^\perp$. The corresponding Gibbons--Hawking space $X_\zeta$ contains a holomorphic sphere $\Sigma$ such that $(\Sigma,\xi_2,\xi_3)$ and any choice of $q \in \im \mathbb{H}$ satisfy the conditions of \autoref{prop:coassociatives_examples}. Furthermore, \autoref{bul:cond_embedding_1} and~\ref{bul:cond_embedding_2} of \autoref{ex:coassociative_model} are satisfied with $H_2=\{1\}$ (as $G/\Lambda = \{1\}$). We therefore obtain embedded coassociative submanifolds in all the critical loci. Pick now any $\xi_1\in \Lambda_S$ such that $\{\xi_1,\xi_2,\xi_3\}$ generate $\Lambda_S$. The coassociative submanifold associated to $(\Sigma,\xi_2,\xi_3)$ and the basepoint $q$ equals (up to reparametrisation) the coassociative submanifold associated to $(\Sigma,\xi_2,\xi_3)$ and the basepoint $q + \xi_1$. The families of coassociative submanifolds constructed in this example are therefore parametrised by  $S^1$ (cf.~\autoref{rem:coassfam}).
\end{example}

\begin{remark}
Joyce~\cite[Examples~3, 4, 5, 6, 15, 16]{Joyce-GeneralisedKummer2} constructs further examples of flat $\G_2$-orbifolds whose transverse singularities are modelled upon $\mathbb{H}/\Gamma_S$ for $\Gamma_S \in \{C_2,C_3\}$. \cite[Examples~4.3,~4.4,~4.9]{DPSDTW-Associatives} describes possible choices for the resolution data and points out holomorphic spheres inside the corresponding Gibbons--Hawking spaces. It is not difficult to check that every singular stratum of these orbifolds admits at least one choice of resolution data such that \autoref{prop:coassociatives_examples} gives rise to an embedded coassociative submanifold in the resolution.
\end{remark}

The following examples all treat $\G_2$-orbifolds constructed in~\cite[Section~5.4.3]{Reidegeld}. A neighbourhood of the singular strata in all these orbifolds can be described by \autoref{ex_G2Model} together with the data in \autoref{fig:Singular1}. For this we list:
\begin{itemize}
\item The diffeomorphism type of the singular strata $S$.
\item The orbifold group $\Gamma_S$ such that the transverse singularity is modelled upon $X_S \coloneqq\mathbb{H}/\Gamma_S$.
\item The generators of the Bieberbach group $G_S$ as follows: Every $G_S$ is generated by the lattice $\Lambda_S = \langle i,j,k \rangle \subset \im \mathbb{H}$. Furthermore, we indicate whether the following two additional generators appear (\cmark = appears, \xmark = does not appear): \[ \left(R_+, \frac{i+k}{2}\right) \qquad \text{and} \qquad \left(R_-,\frac{j}{2}\right)\] where $R_\pm \in \GL(\Lambda_S)\cong \GL_3(\mathbb{Z})$ are given by
\begin{align}\label{eq:R+-}
R_\pm \coloneqq
\begin{pmatrix} 
\pm 1 & 0 & 0 \\
0 & \mp 1 & 0 \\
0 & 0 & -1
\end{pmatrix}.
\end{align}
\item The action $\rho_S \col G_S \to \textup{Isom}(\mathbb{H}/\Gamma_S)$ of the generators of $G_S$ as follows: The lattice $\Lambda_S$ acts trivially in all examples. Furthermore, $(R_+,\frac{i+k}{2})$ and $(R_-,\frac{j}{2})$ act (whenever they appear as generators) via $\rho_S(R_+,\frac{i+k}{2})[q] = [iqi]$ and $\rho_S(R_-,\frac{j}{2})[q] = [jqj]$ for $[q]\in \mathbb{H}/\Gamma_S$, respectively.
\end{itemize}

\begin{table}
\begin{center}
{\renewcommand{\arraystretch}{1.4}
\begin{tabular}{ | c || c | c | c | c |}
\cline{4-5}
\multicolumn{1}{c}{} & \multicolumn{1}{c}{} & \multicolumn{1}{c}{} &  \multicolumn{2}{|c|}{$G_S$ and $\rho_S \col G_S \to \textup{Isom}(\mathbb{H}/\Gamma_S)$} \\
\cline{1-5}
\multicolumn{1}{| c ||}{\parbox[0pt][3em][c]{0cm}{} $\#$ } &$S$ & $\Gamma_S$  & \makecell{$(R_+,\frac{i+k}{2})$ with \\  $\rho_S(R_+,\frac{i+k}{2})[q] = [iqi]$} & \makecell{$(R_-,\frac{j}{2})$ with \\ $\rho_S(R_-,\frac{j}{2})[q] = [jqj]$}  \\
\cline{1-5} \hline 
$1.$ & $T^3/C_2^2$ & $C_2$ & \cmark & \cmark  \\
$2.$ & $T^3/C_2$ & $C_3$ & \xmark & \cmark  \\
$3.$ & $T^3/C_2^2$ & $C_3$ & \cmark & \cmark  \\
$4.$ & $T^3/C_2^2$ & $C_4$ & \cmark & \cmark  \\
$5.$ & $T^3/C_2^2$ & $C_6$ & \cmark & \cmark  \\ 
$6.$ & $T^3/C_2^2$ & $\textup{Dic}_3$ & \cmark & \cmark  \\
\hline
\end{tabular}
}
\end{center}
\caption{Description of those singular strata appearing in~\cite[Section 5.3.4]{Reidegeld} which were not treated in~\cite[Section~4]{DPSDTW-Associatives}. For each stratum the Bieberbach group $G_S$ is generated by $\Lambda_S =\langle i,j,k \rangle \subset \im \mathbb{H}$. Whether the two additional generators  $(R_+,\frac{i+k}{2})$ and $(R_-,\frac{j}{2})$ (for $R_{\pm}$ as in~\eqref{eq:R+-}) appear is indicated (\cmark = appears, \xmark = does not appear). The homomorphism $\rho_S \col G_S \to \textup{Isom}(\mathbb{H}/\Gamma_S)$ maps $\Lambda_S$ to $\textup{Id}$ and the other generators as indicated. A neighbourhood of any singular stratum is then described by \autoref{ex_G2Model} together with the respective data of this table.}\label{fig:Singular1}
\end{table}

\begin{example}\label{ex:Reidegeld_descripition_qdependence}
Reidegeld~\cite[Section~5.3.4]{Reidegeld} constructs an example of a flat $\G_2$-orbifold whose singular strata split into two types. Both types can be described via \autoref{ex_G2Model} together with the data of rows~$2$ and~$3$ in \autoref{fig:Singular1}, respectively. 

These singularities can be resolved by Gibbons--Hawking spaces. This requires a choice of parameter $\zeta\in \Delta^\circ$ (cf. \autoref{sec:ALE}). All parameters such that the $G_S$-action lifts to the Gibbons--Hawking space $X_\zeta$ can be found in \autoref{sec:ReidegeldResolution}.

The parameter $\zeta \coloneqq [-i,0,i]$ works for both types of singularities. The corresponding Gibbons--Hawking space contains two $I_i$-holomorphic spheres which together with $\xi_2 \coloneqq j$, $\xi_3 \coloneqq k$, and any choice of $q \in \im \mathbb{H}$ satisfy the conditions of \autoref{prop:coassociatives_examples}. Thus, the resolution admits coassociative submanifolds in all the critical regions. 

The conditions stated in \autoref{bul:cond_embedding_1} and~\ref{bul:cond_embedding_2} of \autoref{ex:coassociative_model} are satisfied by the above choices. However, the group $H_2$ in \autoref{bul:cond_embedding_2} depends on the value of the basepoint $q$ which we set (without loss of generality) as $q \coloneqq si$ with $s \in \mathbb{R}$ in the following. For the resolution of the singularities described by row~2 we then have $H_2= C_2$ if and only if $s \in \frac{1}{2}\mathbb{Z}$ and for all other values $H_2=\{1\}$. Note further that the coassociatives for $q=si$,$q=(s+1)i$, and $q=-si$ all agree (up to reparametrisation) in $\hat{Y}_t$. The family of coassociative submanifolds that we obtain by varying $q$ (cf. \autoref{rem:coassfam}) is therefore parametrised by the interval $[0,1/2]$ (more precisely, by $S^1/\mathbb{Z}_2$ where $\mathbb{Z}_2$ acts via reflection). For all inner points $s\in (0,1/2)$ the corresponding submanifolds are embedded and for $s\in \{0,1/2\}$ they factor through double cover over embedded rigid coassociatives.  

Similarly, the coassociatives inside the resolution of the singularities described by row~$3$ are embedded for $q \coloneqq si$ with $s \notin \frac{1}{4}\mathbb{Z}$ and factor through a double cover over an embedded submanifold for these critical values. Furthermore, coassociatives for $q=si,q=(s+1/2)i$, and $q=-si$ are (up to reparametrisation) identified in $\hat{Y}_t$. As before, we therefore obtain that the deformation family is given by the interval $[0,1/4]$ with embedded coassociatives for $s \in (0,1/4)$ and double covers for $s \in \{0,1/4\}$.
\end{example}

\begin{example}
Reidegeld~\cite[Section~5.3.4]{Reidegeld} constructs an example of a flat $\G_2$-orbifold whose singular strata split into two types. Both types can be described via \autoref{ex_G2Model} together with the data of rows~$1$ and~$4$ in \autoref{fig:Singular1}, respectively. 

We choose a set of resolution data by a collection of certain Gibbons--Hawking spaces. This requires choices of the parameter $\zeta\in \Delta^\circ$ (cf. \autoref{sec:ALE}). All parameters such that the $G_S$-action lifts to the Gibbons--Hawking space $X_\zeta$ can be found in \autoref{sec:ReidegeldResolution}.

As an example, we pick the following:
\begin{enumerate}
\item[$1.$] $\zeta =[-i,i]$ for strata of type described by row~$1$. The associated Gibbons--Hawking space contains one $I_i$-holomorphic sphere which together with $\xi_2 \coloneqq j$,$\xi_3 \coloneqq k$, and any basepoint $q\in \im \mathbb{H}$ satisfies the conditions of \autoref{prop:coassociatives_examples}. As in \autoref{ex:Reidegeld_descripition_qdependence}, these are embedded for generic choices of $q\in \im \mathbb{H}$ and otherwise factor through a double-cover over an embedded rigid coassociative.
\item[$4.$] $\zeta =[-2i,-i,i,2i]$ for strata of type described by row~$4$. The associated Gibbons--Hawking space contains $3$ $I_i$-holomorphic spheres which together with $\xi_2 \coloneqq j$, $\xi_3 \coloneqq k$, and any $q\in \im \mathbb{H}$ satisfy the conditions of \autoref{prop:coassociatives_examples}. The resulting submanifolds are again embedded for generic $q$ and factor otherwise through double-cover over embedded coassociatives.
\end{enumerate}
\end{example}

\begin{example}
Reidegeld~\cite[Section~5.3.4]{Reidegeld} constructs an example of a flat $\G_2$-orbifold whose singular strata split into four types. All types can be described via \autoref{ex_G2Model} together with the data of rows~$1$, $2$, $3$, and~$5$ in \autoref{fig:Singular1}, respectively. 

All singularities can be resolved by certain Gibbons--Hawking spaces. This requires choices of the parameter $\zeta\in \Delta^\circ$ (cf. \autoref{sec:ALE}). All parameters such that the $G_S$-action lifts to the Gibbons--Hawking space $X_\zeta$ can be found in \autoref{sec:ReidegeldResolution}.

The singular strata described by rows $1$-$3$ have been treated in the previous examples. For the strata of type $5$ we may exemplary pick $\zeta \coloneqq [-3i,-2i,-i,i,2i,3i]$. The corresponding Gibbons--Hawking space contains five $I_i$-holomorphic spheres. Each one of these together with $\xi_2 \coloneqq j,\xi_3 \coloneqq k$, and any $q\in \im \mathbb{H}$ satisfies the conditions of \autoref{prop:coassociatives_examples}. As in \autoref{ex:Reidegeld_descripition_qdependence}, these are generically embedded and factor otherwise through double-cover over embedded coassociatives.
\end{example}

The following example discusses an orbifold where some of the singularities are described by row 6 of \autoref{fig:Singular1}. A resolution thereof requires Kronheimer's construction of ALE spaces (as described in \autoref{Sec:KronheimerALE}). 

Recall from \autoref{rem:ADE} that the group $\textup{Dic}_3$ corresponds to the root system $D_5$. This root system is given by (cf.~\cite[Chapter~VI.4.8]{Bourbaki-Lie4-6}) \[ \Phi = \{\pm e_i \pm e_j \in \mathbb{R}^5 \mid i \neq j \in \{1,\dots,5\} \}\] and one possible choice of simple roots consists of \[ \alpha_i \coloneqq e_i -e_{i+1} \text{ for $i=1,\dots,4$} \quad \text{and} \quad \alpha_5 \coloneqq e_5+e_{4}.\]

\begin{example}\label{Ex:Reidegeld_Dic3}
Reidegeld~\cite[Section~5.3.4]{Reidegeld} constructs an example of a flat $\G_2$-orbifold whose singular strata split into four types. All types can be described via \autoref{ex_G2Model} together with the data of rows~$1$, $3$, $4$, and~$6$ in \autoref{fig:Singular1}, respectively. 

Singularities of types described by rows $1$, $3$, and $4$ have been treated in the previous examples. We therefore focus on the strata determined by row $6$. A resolving ALE space can be constructed via Kronheimer's method and requires a choice of parameter $\zeta \in \Delta^\circ$ (cf. \autoref{Sec:KronheimerALE}). All parameters such that the $G_S$-action lifts to the ALE space $X_\zeta$ can be found in \autoref{sec:ReidegeldResolution}.

For example, the choice $\zeta\coloneqq [4i,3i,2i,i,0]$ leads to an ALE space $X_\zeta$ such that the $G_S$-action lift. We now pick $\tilde{\zeta}\coloneqq (4i,3i,2i,i,0) \in \im \mathbb{H} \otimes \mathbb{R}^5$ as a representative of $\zeta$. The corresponding ALE space $X_{\tilde{\zeta}}$ contains five $I_i$-holomorphic spheres $\Sigma_{\alpha_1},\dots,\Sigma_{\alpha_5}$ (which correspond to the simple roots of the $D5$ root system $\Phi$ described above; cf. \autoref{bul: holomorphic curves in Kronheimer's ALE} of \autoref{Sec:KronheimerALE}). Each one of these together with $\xi_2 \coloneqq j$, $\xi_3 \coloneqq k$, and any choice of $q \in \im \mathbb{H}$ satisfies the conditions of \autoref{prop:coassociatives_examples} and gives therefore rise to a coassociative submanifold.

By \autoref{bul: holomorphic curves in Kronheimer's ALE} of \autoref{Sec:KronheimerALE}, the spheres $\Sigma_{\alpha_1},\dots,\Sigma_{\alpha_5}$ are embedded if there exists no decomposition $\alpha_i = \theta_1 + \theta_2$ into roots $\theta_1,\theta_2 \in \Phi$ such that $\vert \tilde{\zeta}(\theta_1)\vert + \vert \tilde{\zeta}(\theta_2)\vert = \vert \tilde{\zeta}(\alpha_i)\vert$. Direct inspection of $\tilde{\zeta}$ applied to any root $\theta = \pm e_i \pm e_j \in \Phi$ reveals that this is indeed the case.

From \autoref{bul: holomorphic curves in Kronheimer's ALE} and \autoref{bul: Kronheimer Weyl group} of \autoref{Sec:KronheimerALE} and the description of the adjoint action of $\rho_S(R_+,\frac{i+k}{2})$ and $\rho_S(R_-,\frac{j}{2})$ on $\mathbb{R}\Phi$ given in \autoref{sec:ReidegeldResolution} follows that the lifts of $\rho_S(R_+,\frac{i+k}{2})$ and $\rho_S(R_-,\frac{j}{2})$ to $X_{\zeta}$ preserve the spheres $\Sigma_{\alpha_1},\dots,\Sigma_{\alpha_3}$ and interchange the disjoint spheres $\Sigma_{\alpha_4}$ and $\Sigma_{\alpha_5}$. We again set $q \coloneqq si$ for $s\in \mathbb{R}$ as the basepoint and denote by $\iota^q_{\alpha_i} \col M_{\alpha_i} \to \hat{Y}_t$ the immersed coassociative submanifold associated to $(\Sigma_{\alpha_i},j,k)$ and $q$. Up to reparametrisation, these satisfy $\iota^q_{\alpha_4} = \iota^{q+\frac{i}{2}}_{\alpha_5}$ and $\iota_{\alpha_4}^q = \iota_{\alpha_4}^{-q+\frac{i}{2}}$ which comes from applying $\rho_S(R_+\frac{i+k}{2})$ and $\rho_S(R_+R_-,\frac{i+k-j}{2})$, respectively. This implies that $(\Sigma_{\alpha_4},j,k)$ and $(\Sigma_{\alpha_5},j,k)$ give rise to \ita{one} deformation family of coassociative submanifolds inside $\hat{Y}_t$, parametrised by $s \in [-\frac{1}{4},\frac{1}{4}]$. As in \autoref{ex:Reidegeld_descripition_qdependence} we obtain that for the inner values these submanifolds are embedded and at the boundary, they factor through double cover. The coassociatives associated to $\Sigma_{\alpha_1},\dots,\Sigma_{\alpha_3}$ behave as in \autoref{ex:Reidegeld_descripition_qdependence}.  
\end{example}

\begin{remark}
Reidegeld~\cite[Section~5.3.4]{Reidegeld} constructs two further examples of orbifolds whose transverse singularities are modelled on $\mathbb{H}/\Gamma_S$ for $\Gamma_S \in \{ C_2,C_4,\textup{Dic}_2\}$. These are treated in~\cite[Examples~4.5 and~4.6]{DPSDTW-Associatives} and we note that \autoref{prop:coassociatives_examples} produces coassociative submanifolds in all critical loci.
\end{remark}

\begin{remark}
In~\cite{JoyceKarigiannis-Kummer} Joyce and Karigiannis extended the generalised Kummer construction to certain non-flat $\G_2$-orbifolds. If similar estimates as in \autoref{theo:Platt} continue to hold, then it seems plausible that the construction method for coassociative submanifolds presented in the current article can be extended to these new manifolds.
\end{remark}

\begin{remark}
Assume for simplicity the following situation: Let $(Y_0,\phi_0)$ be a flat $\G_2$-orbifold whose singularities are all modelled upon $T^3 \times \mathbb{H}/C_2$ where $T^3 = \im \mathbb{H}/\mathbb{Z}^3$ (this is for example the case in~\cite[Example~3]{Joyce-GeneralisedKummer2}). These singularities can be resolved by Gibbons--Hawking spaces $X_\zeta$ for any choice of parameter $\zeta \coloneqq [-x,x] \in (\im \mathbb{H}\setminus \{0\})^2/\{\pm1\}$. However, in order to apply \autoref{prop:coassociatives_examples}, we need the line $\ell \coloneqq  \mathbb{R}x$ to intersect $\mathbb{Z}^3 \subset \im \mathbb{H}$ (this is precisely the second condition of \autoref{ass:coassociative_data}). The following regards the situation where this condition fails: 

Assume that $x \in \im \mathbb{H} \setminus\{0\}$ is such that the line $\mathbb{R}x \subset \im \mathbb{H}$ is 'irrational' (i.e. does not intersect $\mathbb{Z}^3$). Then one could approximate $x$ by a sequence $(x_n)_{n\in \mathbb{N}} \subset \im \mathbb{H}\setminus\{0\}$ such that all corresponding lines $\mathbb{R}x_n$ are rational (i.e. do intersect $\mathbb{Z}^3$). For each resolution by $T^3 \times X_{\zeta_n}$ with $\zeta_n \coloneqq [-x_n,x_n]$ we obtain a $\phi_t$-coassociative submanifold $M_n \subset \hat{Y}_t$ for $t< T_n$ by \autoref{prop:coassociatives_examples}. However, as $n\to \infty$ we have that $T_n \to 0$. Thus (after rescaling) these coassociatives only converge to a (non-compact) coassociative inside the limiting $\mathbb{R}^3 \times X_\zeta$ for $\zeta=[-x,x]$. One might however hope that once $x_n\to x$ converges sufficiently faster than 
$T_n \to 0$,\footnote{see~\cite{DiophantineApproximation} for an overview on the measure-theoretic properties of irrational numbers that are approximable by rationals with a given rate} then \ita{some instance} of this limiting coassociative is already visible inside the resolution by the irrational $T^3 \times X_\zeta$ shortly before the orbifold limit is reached. 

Unfortunately, \autoref{theo:perturbing_Coassociatives} seems to be of little help when addressing this question. This is because the two $\G_2$-structures $\phi_t(\zeta_n)$ and $\phi_t(\zeta)$ on $\hat{Y}_t$ constructed by resolving respectively with a rational $\zeta_n$ and the irrational $\zeta$ lie in different cohomology classes. The $\phi_t(\zeta_n)$-coassociatives constructed in \autoref{prop:coassociatives_examples} can then not be perturbed further to $\phi_t(\zeta)$-coassociatives because the second condition of \autoref{theo:perturbing_Coassociatives} is violated.
\end{remark}

\appendix
\section{The quadratic estimate}

\label{Sec:quadraticestimate}
This section establishes the quadratic estimate for the map $\mathcal{N}_\mathtt{J}$ in \autoref{prop:linearisation}.
\begin{lemma}
Let $v,w \in \Gamma(TM)$ be vector fields and $\eta \in \Omega^\ell(M)$ be an $\ell$-form. Then the following identities hold for any torsion free connection $\nabla$:
\begin{align*}
\pounds_w \eta &= \nabla_w \eta + \langle \nabla w \wedge \eta \rangle \\
\pounds_v \pounds_w \eta &= \nabla_v\nabla_w \eta + \langle \nabla w \wedge \nabla_v \eta \rangle + \langle \nabla_{v,\cdot}^2 w \wedge \eta \rangle \\
&\qquad + \langle \nabla v \wedge \nabla_w \eta \rangle + \langle \nabla v \wedge \langle \nabla w \wedge \eta \rangle\rangle
\end{align*}
where $\langle \  \cdot \ \wedge\ \cdot \ \rangle \col T^*M\otimes TM \otimes \Lambda^kT^*M \to \Lambda^kT^*M$ contracts the second and third $TM\otimes T^*M \cong \underline{\mathbb{R}}$ component and takes the the wedge product afterwards. Furthermore, $\nabla^2_{v,w} = \nabla_v \nabla_w - \nabla_{\nabla_v w}$ denotes the second covariant derivative.
\end{lemma}
\begin{proof}
Since $\nabla$ is torsion-free, the equality 
\begin{align*}
(\pounds_w \eta)(u_1,\dots,u_k) &= (\nabla_w \eta)(u_1,\dots,u_k) + \sum_i (-1)^{i+1} \eta(\nabla_{u_i}w,u_1,\dots\hat{u}_i,\dots,u_k)
\end{align*}
holds. This is the first identity and the second is proven similarly.
\end{proof}

Recall from \autoref{sec:Perturbing_coassociatives} that $\iota \col M \to Y$ is a coassociative immersion equipped with a tubular neighbourhood $\mathtt{J}\col U \to Y$. Furthermore, let $F_\mathtt{J}$ and $\mathcal{N}_\mathtt{J}$ be defined as in \autoref{prop:linearisation}.

\begin{lemma}\label{lem:secDerEstimate}
Let $u,v,w\in \Gamma(U)$. The second derivative of $F_\mathtt{J}$ can be estimated by
\begin{align*}
\Vert (D_{u}DF_\mathtt{J})(v)(w) \VertH{k} \leq c (1+ \Vert u \VertH{k+1}) \Vert v \VertH{k+1} \Vert w \VertH{k+1}
\end{align*}
where we regard the differential $DF_\mathtt{J}$ as a map from $\Gamma(U)$ to $\Hom(\Omega^2_+(M),\Omega^3(M))$ and accordingly, $(D_uDF_\mathtt{J})(v) \in \Hom(\Omega^2_+(M),\Omega^3(M))$.
\end{lemma}
\begin{proof}
Lift the sections $v,w\in \Gamma_M(U)$ to vector fields $\hat{v},\hat{w} \in \Gamma_U(TU)$ via $\hat{v}(u_m) \coloneqq \frac{\diff}{\diff t} u_m+tv(m)_{\vert t=0}$ where $m\in M$ denotes the basepoint of $u_m$ (and analogously for $\hat{w}$). Denote their respective flows by $\varphi^{\hat{v}}$, and $\varphi^{\hat{w}}$. Then 
\begin{align*}
(D_uDF_\mathtt{J})(v)(w) &= \partial_t\partial_s F_\mathtt{J}(u+tv+sw) = \partial_t \partial_s u^* (\varphi_t^{\hat{v}})^*(\varphi_s^{\hat{w}})^* (\mathtt{J}^*\phi) \\
&= u^*\pounds_{\hat{v}} \pounds_{\hat{w}} (\mathtt{J}^*\phi).
\end{align*}
Thus, $\Vert D_uDF_\mathtt{J}(v)(w) \VertH{k} \leq c_1 \Vert D u \VertH{k} \Vert \pounds_{\hat{w}} \pounds_{\hat{v}} (\mathtt{J}^*\phi) \VertH{k}$. 

The connection on $\Lambda^2_+T^*M$ induces a decomposition of the tangent bundle $TU$ into vertical $V$ and horizontal component $H$. The vertical part of the differential $D u \in \Gamma(\Hom(TM,u^*TU))$ is given (up to the identification of $u^*V$ with $\Lambda^2_+T^*M$) by $\nabla u$ and the horizontal component is up to the identification $u^*H \cong TM$ given by the identity map. Therefore, $\Vert D_uDF_\mathtt{J}(v)(w) \VertH{k} \leq c_2 (1+\Vert \nabla u \VertH{k}) \Vert \pounds_{\hat{w}}\pounds_{\hat{v}} (\mathtt{J}^*\phi) \VertH{k}$.

To estimate the Lie derivative, we invoke the previous lemma. The only two terms that might require an explanation are $\nabla_{\hat{v}} \nabla_{\hat{w}} (\mathtt{J}^*\phi)$ and $\langle \nabla^2_{{\hat{v}},\cdot}{\hat{w}} \wedge (\mathtt{J}^*\phi) \rangle$. The first can be estimated by 
\begin{align*}
\Vert \nabla_{\hat{v}} \nabla_{\hat{w}} (\mathtt{J}^*\phi) \VertH{k} &\leq \Vert i_{\hat{w}}(\nabla_{\hat{v}} \nabla (\mathtt{J}^*\phi)) \VertH{k} + \Vert \nabla_{\nabla_{\hat{v}} {\hat{w}}} (\mathtt{J}^*\phi)) \VertH{k} \\
& \leq c_3 \Vert w \VertH{k} \Vert v \VertH{k} (\Vert \nabla\nabla (\mathtt{J}^*\phi) \VertH{k} + \Vert \nabla(\mathtt{J}^*\phi) \VertH{k}).
\end{align*}
Note that in the second line there is no additional derivative of $\hat{w}$ coming from $\nabla_{\hat{v}}\hat{w}$. This is because $(\nabla_{\hat{v}}\hat{w})(u_m)$ only depends on $v(m)$ and $w(m)$. (In fact, one can define a map $\Phi \col U\times_MU \to TU$ by $(u_1,u_2)\mapsto \nabla_{\hat{u}_1}\hat{u}_2$.)

\noindent Similarly,
\begin{align*}
\Vert \langle \nabla^2_{v,\cdot} w \wedge (\mathtt{J}^*\phi) \rangle \VertH{k} & \leq c_4 (\Vert v \VertH{k} \Vert w \VertH{k} + \Vert v \VertH{k+1} \Vert w \VertH{k+1}) \Vert \mathtt{J}^*\phi \VertH{k}
\end{align*}
which together with with the observation that $\Vert \mathtt{J}^*\phi \VertH{k+2}$ is bounded finishes the proof.
\end{proof}

\begin{proposition}
The quadratic estimate \begin{align*}
\Vert \mathcal{N}_\mathtt{J}(v) - \mathcal{N}_\mathtt{J}(w) \VertH{k} & \leq c (1+\Vert v \VertH{k+1}+\Vert w \VertH{k+1}+\Vert v-w \VertH{k+1}) \\
& \qquad \Vert v-w\VertH{k+1} (\Vert v \VertH{k+1} + \Vert w \VertH{k+1}) 
\end{align*} holds.
\end{proposition}

\begin{proof}
This follows immediately from the previous lemma and
\begin{align*}
\mathcal{N}_\mathtt{J}(v) - \mathcal{N}_\mathtt{J}(w) &= \int_0^1 D_{tv}F_\mathtt{J}(v-w) - D_0F_\mathtt{J}(v-w) + (D_{tv}F_\mathtt{J} - D_{tw}F_\mathtt{J})(w) \diff t \\
&= \int_0^1 \int_0^t (D_{sv}DF_\mathtt{J})(v)(v-w) + (D_{tw+s(v-w)}DF_\mathtt{J})(v-w)(w) \diff s \diff t.
\end{align*}
\end{proof}

\section{Resolution data for Reidegeld's orbifolds}

\label{sec:ReidegeldResolution}

In this section we describe how to construct resolution data for the $\G_2$-orbifolds of~\cite[Section~5.3.4]{Reidegeld} that were used in \autoref{Sec:Examples}. A neighbourhood of the singular strata in all these orbifolds can be described by \autoref{ex_G2Model} using the data from \autoref{fig:Singular1} (cf. \autoref{Sec:Examples}).

The resolution data for singular strata which are described in rows~$1.$-$5.$ of \autoref{fig:Singular1} can be constructed via the Gibbons--Hawking Ansatz (cf. \autoref{sub:GibbonsHawking}) or, equivalently, via Kronheimer's construction (cf. \autoref{Sec:KronheimerALE}). Recall that in order to obtain a smooth manifold we need to choose for either method a parameter $\zeta$ from \[ \Delta^\circ \coloneqq \{ [\zeta_1,\dots,\zeta_N] \in (\im \mathbb{H})^N/S_N \mid \zeta_1+\dots+\zeta_N= 0 \textup{ and $\zeta_i\neq \zeta_j$ for $i \neq j$} \}. \] To lift the action of $G_S$ we need to restrict further to the following sets:
\begin{enumerate}
\item \( 
\begin{aligned}[t](\Delta^\circ)^{R_+,(-R_-)} = &\{ [\zeta_1,R_+\zeta_1] \in \Delta^\circ \mid  \zeta_1 \in \big((\mathbb{R}j)^\perp \cup (\mathbb{R}k)^\perp\big)\setminus \mathbb{R}i  \} \cup \\
  &\{ [-\zeta_1,\zeta_1] \in \Delta^\circ \mid \zeta_1 \in \mathbb{R}i \}
\end{aligned} 
\)
\item \( 
\begin{aligned}[t]
(\Delta^\circ)^{(-R_-)} \hspace{0.52cm} = &\{ [\zeta_1,\zeta_2,-R_- \zeta_2] \in \Delta^\circ \mid \zeta_1 \in (\mathbb{R}j)^\perp, \zeta_2 \notin (\mathbb{R}j)^\perp\} \cup \\
&\{ [\zeta_1,\zeta_2,\zeta_3] \in \Delta^\circ \mid \zeta_1,\zeta_2,\zeta_3 \in (\mathbb{R}j)^\perp\} 
\end{aligned} 
\)
\item \( 
\begin{aligned}[t]
(\Delta^\circ)^{R_+,(-R_-)} = &\{ [\zeta_1,\zeta_2,R_+\zeta_2] \in \Delta^\circ \mid \zeta_1 \in \mathbb{R}i, \zeta_2 \in \big((\mathbb{R}j)^\perp \cup (\mathbb{R}k)^\perp\big)\setminus \mathbb{R}i  \} \cup \\
&\{ [\zeta_1,\zeta_2,\zeta_3] \in \Delta^\circ \mid \zeta_1,\zeta_2,\zeta_3 \in \mathbb{R}i \}
\end{aligned} 
\)
\item \( 
\begin{aligned}[t]
(\Delta^\circ)^{R_+,(-R_-)} = &\{ [\zeta_1,R_+\zeta_1,-R_-\zeta_1,-R_+R_-\zeta_1] \in \Delta^\circ \mid \zeta_1 \notin (\mathbb{R}j)^\perp \cup (\mathbb{R}k)^\perp \} \cup \\
&\{ [\zeta_1,R_+\zeta_1,\zeta_2,R_+\zeta_2] \in \Delta^\circ \mid \zeta_1,\zeta_2 \in \big((\mathbb{R}j)^\perp \cup (\mathbb{R}k)^\perp\big)\setminus \mathbb{R}i\} \cup  \\
&\{ [\zeta_1,\zeta_2,\zeta_3,R_+\zeta_3] \in \Delta^\circ \mid \zeta_1,\zeta_2 \in \mathbb{R}i, \zeta_3 \in \big((\mathbb{R}j)^\perp \cup (\mathbb{R}k)^\perp\big)\setminus \mathbb{R}i\} \cup  \\
&\{ [\zeta_1,\zeta_2,\zeta_3,\zeta_4] \in \Delta^\circ \mid \zeta_1,\zeta_2,\zeta_3,\zeta_4 \in \mathbb{R}i \} 
\end{aligned} 
\)
\item \( 
\begin{aligned}[t]
(\Delta^\circ)^{R_+,(-R_-)} = &\{ [\zeta_1,R_+\zeta_1,\zeta_2,R_+\zeta_2,-R_-\zeta_2,-R_+R_-\zeta_2] \in \Delta^\circ \mid \\
 & \qquad \qquad \zeta_1 \in \big((\mathbb{R}j)^\perp \cup (\mathbb{R}k)^\perp\big)\setminus \mathbb{R}i, \zeta_2 \notin (\mathbb{R}j)^\perp \cup (\mathbb{R}k)^\perp \} \cup \\
&\{ [\zeta_1,\zeta_2,\zeta_3,R_+\zeta_3,-R_-\zeta_3,-R_+R_-\zeta_3] \in \Delta^\circ \mid \\
&\qquad \qquad \zeta_1,\zeta_2 \in \mathbb{R}i, \zeta_2 \notin (\mathbb{R}j)^\perp \cup (\mathbb{R}k)^\perp \} \cup \\
&\{ [\zeta_1,R_+\zeta_1,\zeta_2,R_+\zeta_2,\zeta_3,R_+\zeta_3] \in \Delta^\circ \mid \\
&\qquad \qquad \zeta_1,\zeta_2,\zeta_3 \in \big((\mathbb{R}j)^\perp \cup (\mathbb{R}k)^\perp\big)\setminus \mathbb{R}i\} \cup \\
&\{ [\zeta_1,\zeta_2,\zeta_3,R_+\zeta_3,\zeta_4,R_+\zeta_4] \in \Delta^\circ \mid \\
&\qquad \qquad \zeta_1,\zeta_2 \in \mathbb{R}i,\zeta_3,\zeta_4 \in \big((\mathbb{R}j)^\perp \cup (\mathbb{R}k)^\perp\big)\setminus \mathbb{R}i\} \cup \\
&\{ [\zeta_1,\zeta_2,\zeta_3,\zeta_4,\zeta_5,R_+\zeta_5] \in \Delta^\circ \mid \\
&\qquad \qquad \zeta_1,\zeta_2,\zeta_3,\zeta_4 \in \mathbb{R}i,\zeta_5\in \big((\mathbb{R}j)^\perp \cup (\mathbb{R}k)^\perp\big)\setminus \mathbb{R}i\} \cup \\
&\{ [\zeta_1,\zeta_2,\zeta_3,\zeta_4,\zeta_5,\zeta_6] \in \Delta^\circ \mid \zeta_1,\zeta_2,\zeta_3,\zeta_4,\zeta_5,\zeta_6 \in \mathbb{R}i\}.
\end{aligned} 
\)
\end{enumerate}

To resolve singular strata described in row~$6.$ of \autoref{fig:Singular1} we need Kronheimer's construction as reviewed in \autoref{Sec:KronheimerALE}.

Recall from \autoref{rem:ADE} that the group $\textup{Dic}_3$ corresponds to the root system $D_5$. This root system is given by (cf.~\cite[Chapter~VI.4.8]{Bourbaki-Lie4-6}) \[ \Phi = \{\pm e_i \pm e_j \in \mathbb{R}^5 \mid i \neq j \in \{1,\dots,5\} \}\] and one possible choice of simple roots consists of \[ \alpha_i \coloneqq e_i -e_{i+1} \text{ for $i=1,\dots,4$} \quad \text{and} \quad \alpha_5 \coloneqq e_5+e_{4}.\]

The hyperplane perpendicular to $\theta \coloneqq \pm e_i \pm e_j$ for $i \neq j$ is \[ D_\theta = \{ x \in \mathbb{R}^5 \mid x_i = \pm x_j \} \] (with $+$ if $\theta_i + \theta_j =0$ and $-$ otherwise). Furthermore, the Weyl group $W = C_2^{4}\rtimes S_5$ acts on $\mathbb{R}^5$ by permuting and changing the signs of an even number of coordinates. Thus, in order to obtain a smooth ALE hyperkähler manifold asymptotic to $\mathbb{H}/\textup{Dic}_3$ via Kronheimer's construction, we must choose the value of the moment map from \[ \Delta^\circ = \{ [\zeta_1,\dots,\zeta_5] \in ((\im \mathbb{H})^* \otimes \mathbb{R}^5)/W \mid \zeta_i \neq \pm \zeta_j\text{ for $i\neq j$}\}. \] 

In order to lift the action of $G_S$, we need to restrict further to a value which is invariant under $(\Lambda^2_+\rho_S(g) \otimes \Ad_{C_{\rho_S(g)}})^*$ for any $g\in G_S$. Since conjugation by $R_1 \coloneqq \rho_S(R_{-},\frac{j}{2}) \in N_{\SO(\mathbb{H}}(\textup{Dic}_3)$ preserves all conjugacy classes of $\textup{Dic}_3$, we obtain $\Ad_{C_{R_1}}^*=1$ (cf. \autoref{bul:description-Ad_C_R-in-Kronheimer} in \autoref{Sec:KronheimerALE}). Similarly, conjugation by $R_2 \coloneqq \rho_S(R_+,\tfrac{i+k}{2})\in N_{\SO(\mathbb{H})}(\textup{Dic}_3)$ interchanges precisely two conjugacy classes of $\textup{Dic}_3$ and therefore $\Ad^*_{C_{R_2}}= \sigma_5$, where $\sigma_5 \col \mathbb{R}^5 \to \mathbb{R}^5$ is the reflection $(x_1,\dots,x_5) \mapsto (x_1,\dots,x_4,-x_5)$. In order to lift the action of $G_S$, we therefore need to choose a parameter from the following set: 
\begin{enumerate}
\item[6.]\( 
\begin{aligned}[t]
(\Delta^\circ)^{(R_+\sigma_5),R_-} = &\{[\zeta_1,R_+\zeta_1,R_-\zeta_1,R_+R_-\zeta_1, \zeta_2] \in \Delta^\circ
 \mid \zeta_1 \notin (\mathbb{R}i)^\perp \cup (\mathbb{R}j)^\perp \cup (\mathbb{R}k)^\perp  \\
& \qquad \qquad \textup{ and } \zeta_2 \in \mathbb{R}j \}\cup \\
&\{[\zeta_1,R_a\zeta_1,\zeta_2,R_b\zeta_2,\zeta_3] \in \Delta^\circ
 \mid \\
 & \qquad \qquad\zeta_1,\zeta_2 \in ((\mathbb{R}i)^\perp \cup (\mathbb{R}j)^\perp \cup (\mathbb{R}k)^\perp) \setminus (\mathbb{R}i \cup \mathbb{R}j \cup \mathbb{R}k), \\
 & \qquad\qquad R_a,R_b \in \{R_+,R_-,R_+R_-\}, \textup{ and } \zeta_3 \in \mathbb{R}j \} \cup \\
&\{[\zeta_1,R_a\zeta_1,\zeta_2,t\zeta_2,\zeta_3] \in \Delta^\circ
 \mid \\
&\qquad\qquad\zeta_1 \in ((\mathbb{R}i)^\perp \cup (\mathbb{R}j)^\perp \cup (\mathbb{R}k)^\perp) \setminus (\mathbb{R}i\cup \mathbb{R}j \cup \mathbb{R}k), \\
 &\qquad\qquad R_a \in \{R_+,R_-,R_+R_-\}, \zeta_2 \in \mathbb{R}i \cup \mathbb{R}j \cup \mathbb{R}k, t\in \mathbb{R}\setminus\{-1\}, \\ 
 & \qquad\qquad  \textup{ and } \zeta_3 \in \mathbb{R}j \} \cup \\
&\{[\zeta_1,R_a\zeta_1,\zeta_2,\zeta_3,0] \in \Delta^\circ
 \mid \\
&\qquad\qquad\zeta_1 \in ((\mathbb{R}i)^\perp \cup (\mathbb{R}j)^\perp \cup (\mathbb{R}k)^\perp) \setminus (\mathbb{R}i\cup \mathbb{R}j \cup \mathbb{R}k), \\
 &\qquad\qquad R_a \in \{R_+,R_-,R_+R_-\}, \zeta_2 \in \mathbb{R}i, \textup{ and } \zeta_3 \in \mathbb{R}k \} \cup \\
&\{[\zeta_1,t_1\zeta_1,\zeta_2,t_2\zeta_2,\zeta_3] \in \Delta^\circ
 \mid \zeta_1,\zeta_2 \in \mathbb{R}i \cup \mathbb{R}j \cup \mathbb{R}k, \\
 &\qquad \qquad t_1,t_2 \in \mathbb{R} \setminus \{-1\}, \textup{ and } \zeta_3 \in \mathbb{R}j \} \cup \\
&\{[\zeta_1,t\zeta_1,\zeta_2,\zeta_3,0] \in \Delta^\circ
 \mid \zeta_1 \in \mathbb{R}i \cup \mathbb{R}j \cup \mathbb{R}k, t \in \mathbb{R} \setminus \{-1\}, \zeta_2 \in \mathbb{R}i, \\
& \qquad \qquad \textup{ and } \zeta_3 \in \mathbb{R}k \}.
\end{aligned} 
\)
\end{enumerate}

\bibliography{references}{}
\bibliographystyle{alpha}

\end{document}